\documentclass[12pt,a4paper]{article}
\usepackage{graphicx}
\usepackage{amsmath}
\usepackage[margin=0.90in]{geometry}
\usepackage{nameref}
\usepackage{ragged2e}
\usepackage{hyperref}
\usepackage{physics}
\usepackage[table]{xcolor}
\usepackage{capt-of}    
\usepackage{tabularx}
\usepackage{amssymb} 
\usepackage{amsthm}
\usepackage{enumitem}
\usepackage{mathrsfs}
\usepackage{authblk}  

\newtheorem{theorem}{Theorem}
\newtheorem{corollary}{Corollary}
\newtheorem{assumption}{Assumption}
\newtheorem{definition}[theorem]{Definition}

\newtheorem{lemma}[theorem]{Lemma}

\newtheorem{remark}[theorem]{Remark}

\title{AN ECO-EPIDEMIOLOGICAL MODEL WITH PREY-TAXIS AND \textquotedblleft SLOW\textquotedblright \, DIFFUSION: GLOBAL EXISTENCE, BOUNDEDNESS AND NOVEL DYNAMICS}

\author[1]{Ranjit Kumar Upadhyay$^{*}$}
\author[2]{Rana D. Parshad}
\author[1]{Namrata Mani Tripathi}
\author[3]{Nishith Mohan}

\affil[1]{Department of Mathematics and Computing, Indian Institute of Technology (Indian School of Mines), Dhanbad, Jharkhand, India}
\affil[2]{Department of Mathematics, Iowa State University, Ames, IA, United States of America}
\affil[3]{Department of Mathematics, Rheinland-Pf\"{a}lzische Technische Universit\"{a}t Kaiserslautern-Landau, Kaiserslautern, Germany}

\makeatletter
\renewcommand\maketitle{%
  \begin{center}
    \normalsize
    {\bfseries \@title \par}
    \vskip 1em
    \@author
     \vskip 1em
    \footnotesize
    \textit{*Corresponding author. Email: ranjit@iitism.ac.in}
  \end{center}
}
\makeatother
\makeatletter
\renewcommand\AB@affilsepx{\par\vspace{0.5em}}  
\makeatother
\begin{document}
\maketitle

	\begin{abstract}
In this manuscript, an attempt has been made to understand the effects of prey-taxis on the existence of global-in-time solutions and dynamics in an eco-epidemiological model, particularly under the influence of slow dispersal characterized by the $p$-Laplacian operator and enhanced mortality of the infected prey, subject to specific assumptions on the taxis sensitivity functions. We prove the global existence of classical solutions when the infected prey undergoes random motion and exhibits standard mortality. Under the assumption that the infected prey disperses slowly and exhibits enhanced mortality, we prove the global existence of weak solutions. Following a detailed mathematical investigation of the proposed model, we shift our focus to analyse the stability of the positive equilibrium point under the scenario where all species exhibit linear diffusion, the infected prey experiences standard mortality, and the predator exhibits taxis exclusively toward the infected prey. Within this framework, we establish the occurrence of a steady-state bifurcation. Numerical simulations are then carried out to observe this dynamical behavior. Our results have large scale applications to biological invasions and biological control of pests, under the prevalence of disease in the pest population.
\end{abstract}


    \section{Introduction}
 The intricacies involved in ecological systems are too complex to be captured by the simple dynamical structure typically assumed in classical predator-prey models. Real-world ecosystems comprise complex interplays, such as disease transmission within prey populations, and spatial heterogeneities, which significantly influence the dynamics. These factors can significantly alter population dynamics, leading to unexpected spatial or temporal patterns, and in some cases, may even drive extinction events. Incorporating such complexities into mathematical models not only provides a nuanced perspective on population behaviour but also helps in understanding and predicting ecological outcomes. 
 
 In this context, eco-epidemiology emerges as a multidisciplinary field that explores the interplay between ecological interactions and the dynamics of infectious diseases within populations, particularly in predator-prey systems \cite{anderson1991populations,chattopadhyay1999predator}. While traditional spatial models often simplify movement patterns by assuming purely random diffusion for both predators and prey, such an assumption overlooks the complex and directed behaviors observed in natural ecosystems. In reality, organisms frequently respond to environmental cues that guide their movement in specific directions. For example, chemotaxis (response to chemical gradients), aero-taxis (response to oxygen concentration), and prey-taxis (response to prey density) are key mechanisms that influence spatial distributions by attracting or repelling individuals \cite{ainseba2008reaction}.

The concept of prey-taxis in ecological modeling was first introduced by Kareiva and Odell \cite{kareiva1987swarms} to account for predator aggregation in high prey density regions. This behavior is mathematically represented by incorporating nonlinear diffusion terms into the model. Chakraborty et al.\cite{chakraborty2007predator}  investigated its role in pattern formation within a Paramecium aurelia–Didinium nasutum predator-prey system, while Lee et al. \cite{lee2009pattern} further examined pattern dynamics under various functional responses. Several studies have demonstrated that prey-taxis significantly influences spatiotemporal dynamics in predator-prey systems. For instance, it can enhance biological control by guiding predator movement towards prey, helping stabilize pest populations, or, conversely, induce chaotic dynamics under high sensitivity levels \cite{sapoukhina2003role}. It may also lead to periodic, quasi-periodic, or chaotic oscillations in predator-prey systems, depending on the form of taxis and initial conditions \cite{chakraborty2008effect}. Prey-taxis have further been shown to suppress or promote spatial pattern formation depending on model structure and taxis cue \cite{lee2009pattern}. In steady-state analysis, prey-tactic sensitivity has been found to delay the stability of homogeneous equilibria and facilitate non-constant steady states when combined with Holling type II responses and Neumann boundary conditions \cite{li2014steady}. More recently, research has focused on one-predator–one-prey systems with prey-taxis \cite{wang2017nonconstant}, addressing global existence, steady states, and dynamical behavior. Wu et al. \cite{wu2016global} examined global existence and uniform persistence in predator-prey models with nonlinear prey-taxis, and Guo and Wang \cite{guo2019dynamics} studied the impact of dual prey-taxis mechanisms on dynamics and pattern formation. Their work highlighted how small sensitivity coefficients influence system behavior. Building on this, Wang and Guo \cite{wang2019dynamics} demonstrated that in a two-predator–one-prey model, stationary patterns emerge primarily due to large taxis components rather than diffusion or weak taxis effects.
Despite these advancements, the role of prey-taxis in systems involving both susceptible and infected prey remains largely unexplored. In this work, we aim to fill this gap by examining how dual prey-taxis influences solution existence, spatial dynamics, and pattern formation in eco-epidemiological systems. 

In most population models, species dispersal via the standard Laplacian operator - this is the case of linear diffusion \cite{jungel2010diffusive}. However, there is much motivation for non-linear dispersal in Biology, such as in long range dispersal of birds and insects, Levy flights, models with memory effects and non-linear harvesting problems \cite{chapman2015long, stinga2023fractional, song2019spatiotemporal, jungel2010diffusive}. Such dispersal/diffusion is also applicable in other areas including gas dynamics and kinetics, thin film dynamics, plasma physics and nuclear technology \cite{bonforte2024cauchy, stinga2023fractional}. There are two popular methods to model non-linear dispersal in the local context. These are the porous medium equation \cite{vazquez2007porous} and the non-newtonian filtration equation via the p-laplacian operator \cite{antontsev2015evolution}. They are amongst a class of problems known as ``degenerate" (also ``singular") problems \cite{dibenedetto2012degenerate, levine1984some}. Herein the ``type" or qualitative nature of the PDE changes in certain parts of the domain or certain parameteric regimes. Essentially, the type of a PDE is defined by one or more algebraic relations between the coefficients comprising strict inequalities. If, at certain points of the domain or boundary, these inequalities are weak rather than strict, the equation is called degenerate. An example is the equation $u_{t}=u_{yy} + y u_{x}, \ y \geq 0$ - which is degenerate when $y=0$. Both of these equations possess considerably greater mathematical challenges than the standard Laplacian counterpart - but can yield richer dynamics such as finite time extinction (FTE) \cite{antontsev2015evolution}. Roughly speaking the porous medium equation degenerates when the solution $u \approx 0$, whilst the non-newtonian filtration equation degenerates when $\nabla u \approx 0$ - yielding very different mathematical techniques for their individual analysis. In the current manuscript we will adhere to the latter's technology, to investigate its effect on an eco-epidemiological system of interest. 

The effect of dispersal via the $p$-Laplacian operator has recently been considered in chemotactic systems \cite{wang2023global} as well as Lokta-Volterra (LV) type competition systems \cite{yang2022global, rani2024global}. However, to the best of our knowledge, its effect on the analysis and dynamics of eco-epidemilogical systems has not been considered earlier. Kong and Liu \cite{cong2016degenerate} considered the Cauchy problem for the  case of the $p$-laplacian Keller-Segal model in spatial dimension $n\geq 3$, \cite{cong2016degenerate}, who under small data assumptions prove the existence of global weak solutions. Certain decay estimates for the solution as well as the extinction of solution in finite time is shown. Wang considered the chemotaxis-hapotaxis model with degradation of the chemical signal and $p$-laplacian motion \cite{wang2023global}. Herein global existence of weak solution is proved for $p> 1 + f(n)$, where $f$ depends on spatial dimension $n$, where $f \rightarrow 2, \ n\rightarrow \infty$. Li \cite{li2020global} considered a attraction repulsion chemotaxis problem with $p$-laplacian, and proved global existence of weak solution under certain parametric restrictions in the fast case ($1<p<2$), and for any parametric ranges in the slow case $(p>2)$.
The case of slow diffusion ($p > 2$), for a system of two competitors, both moving towards a chemical signal was first considered by \cite{yang2022global}, wherein a global bounded weak solution was proved for any bounded initial data and any positive range of parameters. This was extended to the fast diffusion case ($1<p<2$), in \cite{rani2024global}, where again global bounded weak solutions were proved for the $1<p<\frac{3}{2}$ case, for any positive data, and for small positive data in the $\frac{3}{2}<p<2$ case. In general with chemotaxis type of mechanisms, the higher spatial dimension problem $(n\geq 3)$, becomes considerably more difficult, and finite time blow-up is often possible, even with the presence of superlinear damping/absorption terms \cite{winkler2018finite}. Thus, one has to relegate analysis to showing the bounded weak solutions.

In the current manuscript, we consider an eco-epidemiological system with prey-taxis, in space dimension $n\leq 2$. The model consists of a reaction–diffusion-taxis partial differential equation (PDE) describing the evolution of a predator population density and two reaction–diffusion PDE governing the evolution of susceptible and infected prey population densities. We also prescribe bounded initial conditions and no-flux boundary conditions. For the infected prey, we generalize the classical Laplacian operator ($\Delta I$) modeling dispersal of the infected prey, to the $p$-Laplacian operator ($\Delta_{p} I = \nabla \cdot (|\nabla I|^{p-2}\nabla I)$, and consider $p \geq 2$. This divides the dispersal of the infected prey into two separate cases, (i) $p=2$, is the classical dispersal case, (ii) $p>2$ is the case wherein the infected individuals disperse ``slower". Furthermore, we also generalize the mortality term of the infected prey ($-\mu I^{\gamma}$) by considering $0 < \gamma \leq 1$. This also divides the mortality of the infected prey into two separate cases, (i) $\gamma=1$, is the classical mortality case, (ii) $0<\gamma < 1$ is the case wherein the infected prey disperse suffer enhanced or increased mortality, via a faster rate to extinction - in this case one can have finite time extinction (FTE) of the infected prey. The case of $0<\gamma < 1$, or non-linear absorption, as it is referred to in the literature for the semi-linear parabolic equation was studied in detail in \cite{bedjaoui2002critical}, where it is shown one can have classical solutions under certain parametric restriction, but in different parametric regimes, one needs to work with weak solutions due to the FTE, which can lead to non-uniqueness of solutions. Absorption type terms for the semi-linear problem have also been recently considered in \cite{parshad2021some, banerjee2025two}. 

The equations for the current model system are considered in the following form:

\begin{equation}
 \begin{cases}
   \label{eq:model1}
      S_t = \delta_1 \Delta S + rS\left( 1 - \dfrac{S+I}{K} \right) - \lambda SI                                              & ~~~ x \in \Omega, t \in (0,T),               \\[10pt]
      I_t = \delta_2 \nabla \cdot (|\nabla I|^{p-2}\nabla I) + \lambda SI - \dfrac{mPI}{a+ I} - \mu I^{\gamma}, \ 2\leq p, \ 0 < \gamma \leq 1.                                                       & ~~~ x \in \Omega, t \in (0,T),               \\[10pt]
     P_t = \delta_3 \Delta P - \chi_1 \nabla \cdot (\xi(P) P \nabla S) - \chi_2\nabla \cdot(\eta(P) P \nabla I) + \dfrac{m I P}{a + I} - dP                    & ~~~ x \in \Omega, t \in (0,T),                   \\[7pt]
     \dfrac{\partial S}{\partial \nu} =  |\nabla I|^{p-2}\dfrac{\partial I}{\partial \nu} = \dfrac{\partial P}{\partial \nu} = 0         & ~~~ x \in \partial \Omega, t \in (0,T),      \\[7pt]
     S(x,0)= S_0(x), I(x,0) = I_0(x), P(x,0) = P_0(x) & ~~~ x \in \Omega.
 \end{cases}
\end{equation}

Here $\Omega$ is a bounded domain in $\mathbb{R}^n(n \geq 1)$ with smooth boundary $\partial \Omega \in C^{2+\beta}(\overline {\Omega})$, where $0<\beta<1, 0<T \leq +\infty$, initial condition $S_0(x), I_0(x), P_0(x) \in C^{2 + \beta}(\overline{\Omega})$ compatible on $\partial \Omega$, the constants $\delta_i, i=1,2,3$ are nonnegative diffusive coefficients and $r,K,\lambda,m,\mu,a,d$ represents positive parameters in ecology \cite{upadhyay2008chaos} and $\nu$ is the outward directional derivative normal to $\partial \Omega $. The continuous time model i.e., the model without spatial considerations was proposed in \cite{upadhyay2008chaos}. An assumption in the formulation of the continuous-time model states that: ``The infected population does not recover or become immune. The predator (bird) population preys only on infected fish population" \cite[Assumption 4]{upadhyay2008chaos}. Even in this scenario where the growth of the predator population is solely influenced by the infected prey population, the predator still exhibits an attraction to the susceptible population and attempts to catch it. This is modeled by the prey-taxis term $- \chi_1 \nabla \cdot (\xi(P)P \nabla S)$. And also, the infected prey population, an attraction exhibited by the predator population, is modeled by the prey-taxis term $-\chi_2\nabla \cdot (\eta(P)P\nabla I)$. where $\chi_i,i=1,2$ is the prey-taxis coefficients. We break up our analysis into two parts. In the first part, we consider the classical situation $p=2$, $\gamma=1$, so we have standard diffusion and linear mortality. In this case, we are able to prove global existence of classical solutions to our proposed model system. In the second part, we consider, $p>2, 0<\gamma<1$, modeling the earlier mentioned situation of the infected prey individuals moving slower, but suffering faster mortality, due to the disease. In this setting, we prove global existence of weak solutions. An important assumption we make on the density dependent chemotactic coefficients is,

\begin{assumption}
\label{assu2}
$\xi(P), \eta(P) \in C^1([0, \infty)), \xi(P) \equiv 0, \eta(P) \equiv 0$ for $P \geq M$ and $|\xi'(P_1) - \xi'(P_2)| \leq L_1|P_1 - P_2|, |\eta'(P_1) - \eta'(P_2)| \leq L_2|P_1 - P_2|$ for $P_1, P_2 \in [0, \infty)$, with $L_1, L_2, M > 0$  i.e. $\xi'(P)$ and $\eta'(P)$ are Lipschitz continuous.
\end{assumption}
Prey taxis is the direct movement of predator $P$ in response to the variation in prey density $N$, here divided into two subclasses $S$ and $I$. The assumption that $\xi(P) \equiv 0, \eta(P) \equiv 0$ for $P \geq M$ ascertains that there is a threshold value over accumulation of $P$. If $P > M$, then the tactic cross-diffusion induced by $\xi(P)$ and $\eta(P)$ vanishes. The biological justification for this assumption can be found in Ainseba et al. \cite{ainseba2008reaction}. Moreover, He and Zheng \cite{he2015global} pointed out that Assumption \ref{assu2} can be regarded as a type of volume filling effect \cite{painter2002volume} for predator-prey models. This is because Assumption \ref{assu2} prevents overcrowding in an interactive predator-prey scenario, a similar role is played by volume filling effect in models incorporating chemotaxis.
Also note that in the ensuing analysis the constants $c_{i}, C_{i}$ can change in value from line to line, and sometimes within the same line if so required. This follows by a simple relabeling of constants, which for pragmatic reasons we do not keep explicit track of always.  

\section{Boundedness and Global Existence}
\subsection{The case $p=2, \gamma=1$}
Our main analytical result for the classical case, that is when $p=2, \gamma=1$ is:

\begin{theorem}
\label{thm:main_result}
  Let $\Omega \subset \mathbb{R}^n (n \geq 1)$ be a bounded domain with a smooth boundary $\partial \Omega$. Then if the initial conditions are such that $(S_0, I_0, P_0) \in (W^{1, p}(\Omega))^3$ with $S_0, I_0, P_0 \geq 0$ and $p > n$ holds true and Assumption \ref{assu2} holds true.  Then there exists a global classical solution $(S(x, t), I(x, t), P(x, t))$ to \eqref{eq:model1}, with $(S, I, P) \in (C([0, \infty); W^{1, p}(\Omega)) \cap (C^{2, 1}(\bar{\Omega} \times (0, \infty)))^3$ and $(S(x, t), I(x, t), P(x, t))$ are uniformly bounded in $\Omega \times (0, \infty)$.
\end{theorem}

\begin{lemma}
  \label{lem:local_existence}
  Assuming that the initial data satisfies $(S_0, I_0, P_0) \in (W^{1, p}(\Omega))^3$ with $S_0, I_0, P_0 \geq 0$ and $p > n$. Then the following hold true
  \begin{enumerate}
    \item There exists a $T_{\max} > 0$ and nonnegative functions $S(x, t), I(x, t)$ and $P(x, t)$ each belonging in $C(\bar{\Omega} \times [0, T_{\max})) \cap C^{2, 1}(\bar{\Omega} \times [0, T_{\max}))$ such that the triple $(S(x, t), I(x, t), P(x, t))$ solves \eqref{eq:model1} classically in $\Omega \times (0, T_{\max})$. Moreover, the first solution component satisfies
    \begin{align}
      \label{linfty} 0 \leq S(x, t) \leq K
    \end{align}
    for all $(x, t) \in \Omega \times (0, T_{\max})$.
    \item The total mass of $S(x, t), I(x, t)$ and $P(x, t)$ satisfies
    \begin{equation}
      \label{L1bound}
      \int_\Omega (S(x, t) + I(x, t) + P(x, t))dx \leq C \quad \text{for all} ~ t \in (0, T_{\max}).
    \end{equation}
     \item The second solution component satisfies
     \begin{equation}
       \label{linfyI}
       \|I(\cdot, t)\|_{L^\infty(\Omega)} \leq K_1
     \end{equation}
     for all $(x, t) \in \Omega \times (0, T_{\max})$.
    \item For each $T > 0$ there exists a $A_0(T) > 0$ such that
    \begin{equation}
      \label{extensibiliy_criteria}
      \|S(x, t)\|_{L^\infty(\Omega)} + \|I(x, t)\|_{L^\infty(\Omega)} + \|P(x, t)\|_{L^\infty(\Omega)} \leq A_0(T), \quad 0 < T < \min\{T, T_{\max}\},
    \end{equation}
    where $A_0(T)$ depends upon $T$ and $\|S_0, I_0, P_0\|_{W^{1, p}(\Omega)}$, then $T_{\max} = \infty$.
  \end{enumerate}
\end{lemma}
  \begin{proof}
    \begin{enumerate}
      \item The local existence of solutions for system \eqref{eq:model1} can be established by applying the abstract existence theory of Amman \cite{amann1990dynamic}. Let $U = (S, I, P)^T$ then we can write system \eqref{eq:model1} as
      \begin{equation}
    \begin{cases}
        \dfrac{\partial U}{\partial t} = \nabla \cdot (M(U)\nabla U) + F(U),                           & \quad x \in \Omega ,~~ t > 0,                 \\[7pt]
        \dfrac{\partial U}{\partial \nu} = (0, 0)^T,                                                   & \quad x \in \partial \Omega ,~~ t > 0,        \\[7pt]
        U(x,0) = (S_0,I_0,P_0)^T,                                                                      & \quad x \in \Omega,
    \end{cases}
\end{equation}
where \begin{equation*}
     M(U)=\begin{pmatrix}
    \delta_1             & 0 & 0\\[5pt]
    0 & \delta_2 & 0 \\[5pt]
    -\chi_1\xi(P)P & -\chi_2 \eta(P)P & \delta_3
\end{pmatrix} ~ \text{and}~
F(U)=\begin{pmatrix}
 rS( 1 - \tfrac{S+I}{K}) - \lambda SI                         \\[5pt]
 \lambda SI - \tfrac{mPI}{a+ I} - \mu I           \\[5pt]
 \tfrac{m I P}{a + I} - dP \\
\end{pmatrix}.
\end{equation*}
Since the matrix $M(U)$ is positive definite, system \eqref{eq:model1} is normally parabolic and we claim that it admits classical solutions locally \cite[Theorem 7.3]{amann1990dynamic}. Moreover, there also exists a maximal existence time $T_{\max} > 0$ such that system \eqref{eq:model1} admits the triple $(S(x, t), I(x, t), P(x, t)) \in (C(\bar{\Omega} \times [0, T_{\max})) \cap C^{2, 1}(\bar{\Omega} \times [0, T_{\max})))^3$, with $S \geq 0, I \geq 0, P \geq 0$ as its local unique solution.

To prove the next part, observe that by utilizing the non-negativity of the solution components asserted earlier, we can have from the first equation of system \eqref{eq:model1}
\begin{equation*}
  \begin{cases}
    S_t = \delta_1 \Delta S + r\frac{S}{K}(K - S) - \frac{r S I}{K} - \lambda S I \leq \delta_1 \Delta S + r\frac{S}{K} (K - S), & \quad x \in \Omega, t > 0            \\[5pt]
    \frac{\partial S}{\partial \nu} = 0, & \quad x \in \partial \Omega, t > 0,                                                                                          \\[5pt]
    S(x, 0) = S_0(x) & \quad x \in \Omega,
  \end{cases}
\end{equation*}
applying the parabolic maximum principle to the above subsystem yields \eqref{linfty}.

\item Spatially integrating and adding the first three equations of system \eqref{eq:model1}, keeping in mind the nonnegativity of solution components result in
\begin{align}
\label{l1_bound1}
  \nonumber  \dfrac{d}{dt} \left\{\int_{\Omega} S  + \int_{\Omega} I + \int_{\Omega} P \right\}dx & = r \int_{\Omega}Sdx - \dfrac{r}{K}\int_{\Omega} S^{2}dx - \dfrac{r}{K}\int_{\Omega} SIdx - \mu \int_{\Omega}Idx\\[5pt]
 \nonumber & - d \int_{\Omega}Pdx\\[5pt]
& \leq 2 r \int_\Omega S dx - r \int_\Omega S dx - \mu \int_{\Omega}Idx - d \int_{\Omega}Pdx \\
\nonumber &- \dfrac{r}{K}\int_{\Omega} S^{2}dx
\end{align}
for all $t \in (0, T_{\max})$. By Young's inequality, we can have $2 r \int_\Omega S \leq \tfrac{r}{K} \int_\Omega S^2 + r K |\Omega|$, inserting this in  \eqref{l1_bound1} results in
\begin{align}
\label{l1_bound2}
    \dfrac{d}{dt} \left\{\int_{\Omega} S  + \int_{\Omega} I + \int_{\Omega} P \right\}dx \leq &- r \int_\Omega S dx - \mu \int_{\Omega}I dx- d \int_{\Omega}P dx + r K |\Omega| 
\end{align}
for all $t \in (0, T_{\max})$. If we let $z(t) = \int_{\Omega} S  + \int_{\Omega} I + \int_{\Omega} P, t \in (0, T_{\max})$, then from \eqref{l1_bound2} we observe that $z(t)$ satisfies an ordinary differential equation of the form  $z'(t) + c_1 z(t) \leq c_2, c_1 = \min\{r, \mu, d\}$ and $c_2 = r K |\Omega|$ . By \cite[Lemma 3.4]{stinner2014global} we can directly have
\begin{equation*}
  z(t) \leq C = \max \left\{z(0) , \frac{c_2}{c_1}\right\}  \quad ~ \text{for all} ~ t \in (0, T_{\max}).
\end{equation*}
\item To prove \eqref{linfyI}, we combine a variation-of-constants representation for the second equation of  system \eqref{eq:model1} with standard smoothing estimates for the Neumann heat semigroup \cite[Lemma 1.3]{winkler2010aggregation} and the $L^1$ bound of $I$ established in \eqref{L1bound}  to see that with some positive constants $C_2$ and $C_3$,
    \begin{align}
      \label{linfyI11}
      \nonumber \|I(\cdot, t)\|_{L^\infty(\Omega)} & = \sup_{\Omega} \left(e^{t(\delta_2 \Delta - \mu)}I_0 + \int_0^t  e^{(t - s)(\delta_2 \Delta - \mu)} \lambda S I ds -  \int_0^t  e^{(t - s)(\delta_2 \Delta - \mu)} \frac{m P I}{a + I} ds\right)        \\[5pt]
      \nonumber & \leq \sup_{\Omega} \left(e^{t(\delta_2 \Delta - \mu)}I_0\right) + \sup_{\Omega} \left(\lambda \int_0^t e^{(t - s)(\delta_2 \Delta - \mu)}  S I ds \right)               \\[5pt]
      \nonumber & \leq \|e^{t(\delta_2 \Delta - \mu)}I_0\|_{L^\infty(\Omega)} + \lambda \int_0^t \| e^{(t - s)(\delta_2 \Delta - \mu)}  S I \|_{L^\infty(\Omega)}ds                       \\[5pt]
      \nonumber & \leq \|e^{t(\delta_2 \Delta - \mu)}I_0\|_{L^\infty(\Omega)} + \lambda K C_1\int_0^t\left(1 + \delta_2^{-\frac{n}{2}}(t - s)^{-\frac{n}{2}}\right)e^{-\lambda_1 \delta_2(t - s)} \|I(\cdot, s)\|_{L^1(\Omega)}  \\[5pt]
      & \leq \|e^{t(\delta_2 \Delta - \mu)}I_0\|_{L^\infty(\Omega)} + C_2 \int_0^t (1 + s^{- \frac{n}{2}})e^{-C_3 s} ds \leq C_4
    \end{align}
for all $ t \in (0, T_{\max})$, thus establishing \eqref{linfyI}. Here, $\lambda_1 > 0$ is the first nonzero eigenvalue of the $- \Delta$ operator with zero-flux boundary.

\item As system \eqref{eq:model1} is a lower triangular system, \eqref{extensibiliy_criteria} follows from \cite[Theorem 15.5]{amann1993nonhomogeneous}.
\end{enumerate}
\end{proof}

\begin{lemma}
\label{uniform_bounded}
  Given any $(S_0, I_0, P_0) \in (W^{1, p}(\Omega))^3$ with $S_0, I_0, P_0 \geq 0$ and $p > n$, then the solution  $(S(x, t), I(x, t), P(x, t)) \in (C(\bar{\Omega} \times [0, T_{\max})) \cap C^{2, 1}(\bar{\Omega} \times [0, T_{\max})))^3$ of \eqref{eq:model1} is bounded i.e.,
  \begin{equation}
  \label{bounded1}
    \|S(x, t)\|_{L^\infty(\Omega)} + \|I(x, t)\|_{L^\infty(\Omega)} + \|P(x, t)\|_{L^\infty(\Omega)} \leq A_0(T).
  \end{equation}
\end{lemma}
\begin{proof}
  Since we have already established the $L^\infty(\Omega)$ bounds for the components $S(x, t)$ and $I(x, t)$ in \eqref{linfty} and \eqref{linfyI}, to prove \eqref{bounded1}, we only need to demonstrate that $P(x, t)$ also admits an $L^\infty(\Omega)$ bound. Multiplying the first equation in system \eqref{eq:model1} by $P^{q -1} (q > 2)$ and integrating by parts over $\Omega$ keeping in mind the fact that $\frac{I}{a + I} \leq 1$ for $I > 0$, yields
  \begin{align}
  \label{linfb1}
  \nonumber \dfrac{1}{q}\dfrac{d}{dt}\int_{\Omega}P^q dx & + \delta_3(q - 1)\int_{\Omega}P^{q - 2}|\nabla P|^2dx  \leq \chi_1 (q - 1) \int_\Omega P^{q - 1} \xi(P) \nabla P \cdot \nabla Sdx                  \\[5pt]
  & + \chi_2 (q - 1) \int_\Omega P^{q - 1} \eta(P) \nabla P \cdot \nabla Idx + (m - d) \int_\Omega P^qdx.
  \end{align}
  We will first handle the two taxis terms on the right-hand side of \eqref{linfb1}, by Young's inequality
  \begin{align}
    \label{linfb2} \chi_1 (q - 1) \int_\Omega P^{q - 1} \xi(P) \nabla P \cdot \nabla S dx  \leq \dfrac{\delta_3(q - 1)}{4} \int_\Omega P^{q - 2} |\nabla P|^2dx + \dfrac{\chi_1^2 (q - 1)}{\delta_3} \int_\Omega P^q \xi(P)^2 |\nabla S|^2dx            \\[5pt]
    \label{linfb3} \chi_2 (q - 1) \int_\Omega P^{q - 1} \eta(P) \nabla P \cdot \nabla S dx \leq \dfrac{\delta_3(q - 1)}{4} \int_\Omega P^{q - 2} |\nabla P|^2 dx + \dfrac{\chi_2^2 (q - 1)}{\delta_3} \int_\Omega P^q \eta(P)^2 |\nabla I|^2 dx.
  \end{align}
  Inserting \eqref{linfb2} and \eqref{linfb3} in \eqref{linfb1} results in
   \begin{align}
  \label{linfb4}
  \nonumber \dfrac{1}{q}\dfrac{d}{dt}\int_{\Omega}P^qdx & + \frac{\delta_3(q - 1)}{2}\int_{\Omega}P^{q - 2}|\nabla P|^2dx  \leq \dfrac{\chi_1^2 (q - 1)}{\delta_3} \int_\Omega P^q \xi(P)^2 |\nabla S|^2dx                  \\[5pt]
  & +  \dfrac{\chi_2^2 (q - 1)}{\delta_3} \int_\Omega P^q \eta(P)^2 |\nabla I|^2dx + (m - d) \int_\Omega P^qdx.
  \end{align}
 By Assumption \ref{assu2} we can have $\xi(P) \leq M_1$ and $\eta(P) \leq M_1$ because both of them belong to $C^1$ and $\xi, \eta \equiv 0$ for $P \geq M$. Hence, we have
 \begin{align}
  \label{linfb5}
  \nonumber \dfrac{1}{q}\dfrac{d}{dt}\int_{\Omega}P^q dx & + \frac{\delta_3(q - 1)}{2}\int_{\Omega}P^{q - 2}|\nabla P|^2 dx \leq \dfrac{\chi_1^2 M^q M_1^2 (q - 1)}{\delta_3} \int_\Omega |\nabla S|^2 dx                  \\[5pt]
  & +  \dfrac{\chi_2^2 M^q M_1^2(q - 1)}{\delta_3} \int_\Omega|\nabla I|^2 dx + (m - d) \int_\Omega P^q dx.
  \end{align}
 We multiply the first equation of \eqref{eq:model1} by $-\Delta S$, and integrate by parts over $\Omega$ to have
 \begin{align}
  \label{linfb6}
   \nonumber & \frac{1}{2}\frac{d}{dt} \int_\Omega |\nabla S|^2 dx + \delta_1 \int_\Omega |\Delta S|^2 dx = r \int_\Omega |\nabla S|^2 dx + \frac{r}{K} \int_\Omega S^2 \Delta S dx \\[5pt]
   & + \frac{r}{K} \int_\Omega S I \Delta S dx + \lambda \int_\Omega S I \Delta S dx.
 \end{align}
 We can handle the last three terms in the right-hand side of \eqref{linfb6} by Young's inequality and boundedness of $S$ and $I$ (for simplicity we have taken the $L^\infty(\Omega)$ bound of both $S$ and $I$ as $K$ in this proof) , $\tfrac{r}{K}\int_\Omega S^2 \Delta S dx \leq \tfrac{\delta_1}{6} \int_\Omega |\Delta S|^2 dx + \tfrac{3}{2\delta_1} r^2 K^2 |\Omega|, \frac{r}{K} \int_\Omega S I \Delta S dx \leq \tfrac{\delta_1}{6} \int_\Omega |\Delta S|^2 dx + \tfrac{3}{2\delta_1} r^2 K^2 |\Omega|$ and $ \lambda \int_\Omega S I \Delta S dx \leq \tfrac{\delta_1}{6}  \int_\Omega |\Delta S|^2 dx + \tfrac{3}{2\delta_1} \lambda^2 K^2 |\Omega|$. Inserting these in \eqref{linfb6}, result in the following
 \begin{equation}
    \label{linfb7}
   \frac{1}{2}\frac{d}{dt} \int_\Omega |\nabla S|^2 dx + \frac{\delta_1}{2} \int_\Omega |\Delta S|^2 dx \leq  r \int_\Omega |\nabla S|^2 dx + \frac{3}{\delta_1} K^2 |\Omega| (r^2 + \tfrac{\lambda^2}{2}).
 \end{equation}
Now, multiplying the second equation of system \eqref{eq:model1} by $-\Delta I$, and integrating by parts over $\Omega$ results in
\begin{align}
 \label{linfb8}
  \nonumber \frac{1}{2} \frac{d}{dt} &\int_\Omega |\nabla I|^2 dx + \delta_2 \int_\Omega |\Delta I|^2 dx  = - \lambda \int_\Omega (SI) \Delta I dx + m \int_\Omega \frac{I}{a + I} P \Delta I  dx + \mu \int_\Omega I \Delta I  dx \\[5pt]
  & = \lambda \int_\Omega I \nabla S \cdot \nabla I dx + \lambda \int_\Omega S |\nabla I|^2 dx +  m \int_\Omega \frac{I}{a + I} P \Delta I dx - \mu \int_\Omega |\nabla I|^2 dx.
 \end{align}
 Utilizing Young's inequality and the boundedness of $S$ and $I$ allow to have
 \begin{align}
   \label{linfb9} & \lambda \int_\Omega I \nabla S \cdot \nabla I dx \leq 2 \mu \int_\Omega |\nabla I|^2 dx + \frac{\lambda^2 K^2}{8 \mu} \int_\Omega |\nabla S|^2 dx,        \\[5pt]
   \label{1linfb9}& m \int_\Omega \frac{I}{a + I} P \Delta I dx\leq \frac{\delta_2}{2} \int_\Omega |\Delta I|^2  dx+ \frac{m^2}{2 \delta_2} \int_\Omega P^2 dx.
 \end{align}
Inserting \eqref{linfb9} and \eqref{1linfb9} in \eqref{linfb8} results in the following
\begin{equation}
  \label{2linfb9}
  \frac{1}{2} \frac{d}{dt} \int_\Omega |\nabla I|^2 dx + \frac{\delta_2}{2} \int_\Omega |\Delta I|^2 dx \leq (\mu + \lambda K) \int_\Omega |\nabla I|^2 dx + \frac{\lambda^2 K^2}{8 \mu} \int_\Omega |\nabla S|^2 dx + \frac{m^2}{2 \delta_2} \int_\Omega P^2 dx.
\end{equation}
Note that, $\int_\Omega \frac{\delta_3(q - 1)}{2}\int_{\Omega}P^{q - 2}|\nabla P|^2 dx = \frac{2 \delta_3 (q - 1)}{q^2} \int_\Omega |\nabla P^{\frac{q}{2}}|^2 dx, q > 2$, then from \eqref{linfb5}, \eqref{linfb7}, \eqref{2linfb9} and Young's inequality we can have a $c_1 > 0$ such that
\begin{align}
  \label{3linfb9}
  \nonumber & \dfrac{1}{q}\dfrac{d}{dt}\int_{\Omega}P^q dx + \frac{d}{dt} \int_\Omega |\nabla S|^2 dx  + \frac{d}{dt} \int_\Omega |\nabla I|^2  dx + \frac{2 \delta_3 (q - 1)}{q^2} \int_\Omega |\nabla P^{\frac{q}{2}}|^2 dx + \delta_1 \int_\Omega |\Delta S|^2 dx \\[5pt]
  \nonumber & + \delta_2 \int_\Omega |\Delta I|^2 dx \leq (m - d + 1) \int_\Omega P^q dx + \left(\dfrac{\chi_1^2 M^q M_1^2 (q - 1)}{\delta_3} + 2 r + \frac{\lambda^2 K^2}{4 \mu}\right) \int_\Omega |\nabla S|^2 dx  \\[5pt]
  & +  \left(\dfrac{\chi_2^2 M^q M_1^2 (q - 1)}{\delta_3} + 2 \mu + 2 \lambda K\right) \int_\Omega |\nabla I|^2 dx + c_1.
\end{align}
The Sobolev interpolation inequality in conjunction with the boundedness of $S$ and $I$ allow us to have two positive constant $\epsilon_1$ and $\epsilon_2$ such that
\begin{align}
  \label{sob1}
   & \int_{\Omega}|\nabla S|^2 dx \leq \epsilon_1 \int_{\Omega}|\Delta S|^2 dx + c_2\int_{\Omega}|S|^2 dx \leq \epsilon_1 \int_{\Omega}|\Delta S|^2 dx + c_3,  \\[5pt]
   \label{sob2}
   & \int_{\Omega}|\nabla I|^2 dx\leq \epsilon_1 \int_{\Omega}|\Delta I|^2 dx + c_4\int_{\Omega}|I|^2  dx \leq \epsilon_2 \int_{\Omega}|\Delta I|^2 dx + c_5.
\end{align}
with $c_3, c_5 > 0$. The Gagliardo-Nirenberg inequality allow us to have
\begin{equation*}
    \int_{\Omega}P^q dx \leq c_6 \|\nabla P^{q/2}\|^{2\theta}_{L^2(\Omega)}\|P^{q/2}\|^{2(1-\theta)}_{L^{2/q}(\Omega)} + c_6\|P^{q/2}\|^{2}_{L^{2/q}(\Omega)},
    \end{equation*}
with $0 < \theta = \dfrac{nq-n}{nq-n+2} < 1$ and $c_6 > 0$. By Young's inequality,
\begin{equation*}
    \int_{\Omega} P^q dx \leq \epsilon_3 \|\nabla P^{q/2}\|^2_{L^2(\Omega)} + c_7 \|P^{q/2}\|^2_{L^{2/q}(\Omega)} = \epsilon_3 \|\nabla P^{q/2}\|^2_{L^2(\Omega)} + c_7 \|P\|^{q}_{L^1(\Omega)},
\end{equation*}
for any $\epsilon_3 > 0$, with $c_7 > 0$. From \eqref{L1bound}, $\|P\|_{L^1(\Omega)} \leq C$, therefore,
\begin{equation}
  \label{eqbdd5}
    \int_{\Omega}P^q dx \leq \epsilon_3 \|\nabla P^{q/2}\|^2_{L^2(\Omega)} + c_8,
\end{equation}
with some $c_8 > 0$. Now, fix $\epsilon_1, \epsilon_2$ and $\epsilon_3$ such that $(\tfrac{\chi_1^2 M^q M_1^2 (q - 1)}{\delta_3} + 2 r + \tfrac{\lambda^2 K^2}{4 \mu}) \epsilon_1 = \tfrac{\delta_1}{2}$, $(\tfrac{\chi_2^2 M^q M_1^2 (q - 1)}{\delta_3} + 2 \mu + 2 \lambda K)\epsilon_2 = \frac{\delta_2}{2}$ and $(m + 1)\epsilon_3 = \tfrac{2\delta_3(q - 1)}{q^2}$. From \eqref{3linfb9}-\eqref{eqbdd5}, we can directly have
\begin{align}
\label{sob7}
  \nonumber & \frac{1}{q} \dfrac{d}{dt}\int_{\Omega}P^q dx+ \frac{d}{dt} \int_\Omega |\nabla S|^2 dx  + \frac{d}{dt} \int_\Omega |\nabla I|^2 dx \leq - d \int_\Omega P^q dx - (\tfrac{\chi_1^2 M^q M_1^2 (q - 1)}{\delta_3} + 2 r \\[5pt] 
  &  + \tfrac{\lambda^2 K^2}{4 \mu}) \int_\Omega |\nabla S|^2 dx - (\tfrac{\chi_2^2 M^q M_1^2 (q - 1)}{\delta_3} + 2 \mu + 2 \lambda K) \int_\Omega |\nabla I|^2 dx + c_9,
\end{align}
with some $c_9 > 0$. Let
\begin{equation*}
  v(t) :=  \frac{1}{q} \int_{\Omega}P^q dx + \int_\Omega |\nabla S|^2 dx  +  \int_\Omega |\nabla I|^2 dx, \quad t > 0.
\end{equation*}
Then from \eqref{sob7}, we see that $v(t)$ satisfies a ordinary differential equation of the type
\begin{align*}
  v'(t) \leq - c_{10} v(t) + c_9 \quad ~\text{for all}~ t > 0,
\end{align*}
with $c_{10} = \min\{dq, (\tfrac{\chi_1^2 M^q M_1^2 (q - 1)}{\delta_3} + 2 r + \tfrac{\lambda^2 K^2}{4 \mu}) , (\tfrac{\chi_2^2 M^q M_1^2 (q - 1)}{\delta_3} + 2 \mu + 2 \lambda K)\}$, hence
\begin{align*}
  v(t) \leq \max\left\{v(0), \frac{c_9}{c_{10}}\right\} \quad ~\text{for all}~ t > 0.
\end{align*}

 As we are now equipped with a $L^q(\Omega) (q > 2)$ bound for $P$, we will now proceed to find $\|\nabla S\|_{L^q(\Omega)}$ and $\|\nabla I\|_{L^q(\Omega)} (q > 2)$. To this end
 \begin{equation*}
  f(S, I)= rS\left(1- \dfrac{S + I}{K} \right) - \lambda SI.
 \end{equation*}
 It follows from above that
        \begin{equation*}
            \sup_{t>0}\|f\|_{L^q(\Omega)} \leq A_2.
        \end{equation*}
 The variation-of-constants formula applied to the first equation in  system \eqref{eq:model1} results in
        \begin{equation}
            S(.,t) = e^{\delta_1 t \Delta }S_0 + \int^{t}_{0}f(S(s),I(s))ds, ~~~ t > 0.
        \end{equation}
 Using  standard smoothing estimates for the Neumann heat semigroup \cite[Lemma 1.3]{winkler2010aggregation}  we can conclude that
        \begin{align}
  \nonumber \|\nabla S\|_{L^q(\Omega)}& \leq \|\nabla e^{\delta_1 t      \Delta}S_0\|_{L^q(\Omega)} + \int^{t}_{0}\|\nabla e^{\delta_1(t -s)}f(S(s),I(s))\|_{L^q(\Omega)}ds \\
  \nonumber & \leq 2k_2e^{-\lambda_1 t}\|\nabla S_0\|_{L^q(\Omega)} + k_1 \int^{t}_{0}(1 + \delta^{-1/2}_1(t-s)^{-1/2})e^{-\lambda_1 \delta_1(t-s)}\|f(s)\|_{L^q(\Omega)}ds \\
  \nonumber & \leq 2k_2e^{-\lambda_1 t} \|\nabla S_0\|_{L^q(\Omega)} + k_1A_2 \int^{t}_{0}(1  + \delta^{-1/2}_1 s^{-1/2})e^{-\lambda_1 \delta_1 s}ds\\
  \label{nabla_S_bound} & \leq 2k_2 \|\nabla S_0\|_{L^q(\Omega)} + k_1A_2 \left(\dfrac{1}{\lambda_1 \delta_1} + \delta^{-1/2}_1\left( 2 + \dfrac{1}{\lambda_1 \delta_1}\right)\right) ~~\text{for all}~ t > 0.
   \end{align}
 Again, define
\begin{equation*}
    g(S,I,P)= \lambda SI - \dfrac{mPI}{a+ I} - \mu I.
\end{equation*}
It follows from above that there is $A_3>0$ such that
        \begin{equation*}
            \sup_{t>0}\|g\|_{L^q(\Omega)} \leq A_3 < +\infty
        \end{equation*}
        The variation-of-constants formula for $I$ yields
        \begin{equation}
            I(.,t) = e^{\delta_2 t \Delta }I_0 + \int^{t}_{0}g(S(s),I(s),P(s))ds, ~~~ t>0.
        \end{equation}
In a similar manner, as above we can have
        \begin{align}
  \nonumber & \|\nabla I\|_{L^q(\Omega)} \leq \|\nabla e^{\delta_2 t      \Delta}I_0\|_{L^q(\Omega)} + \int^{t}_{0}\|\nabla e^{\delta_2(t -s)}g(S(s),I(s),P(s))\|_{L^q(\Omega)}ds \\
  \nonumber & \leq 2k_4e^{-\lambda_2 t}\|\nabla I_0\|_{L^q(\Omega)} + k_3 \int^{t}_{0}(1 + \delta^{-1/2}_2(t-s)^{-1/2})e^{-\lambda_2 \delta_2(t-s)}\|g(s)\|_{L^q(\Omega)}ds \\
  \nonumber & \leq 2k_4e^{-\lambda_2 t} \|\nabla I_0\|_{L^q(\Omega)} + k_3A_3 \int^{t}_{0}(1  + \delta^{-1/2}_2 s^{-1/2})e^{-\lambda_2 \delta_2 s}ds\\
  \label{nabla_I_bound} & \leq 2k_4 \|\nabla I_0\|_{L^q(\Omega)} + k_3A_3 \left(\dfrac{1}{\lambda_2 \delta_2} + \delta^{-1/2}_2\left( 2 + \dfrac{1}{\lambda_2 \delta_2}\right)\right) ~~\text{for all}~t > 0.
   \end{align}

Equipped with the $L^q(\Omega)$ estimate of $P$, \eqref{nabla_S_bound} and \eqref{nabla_I_bound} we can apply the Moser iterative technique \cite[Lemma A.1]{tao2012boundedness} to prove global boundedness of the solution component $P$.
\end{proof}

\begin{proof}[Proof of Theorem \ref{thm:main_result}]
  Taken together \eqref{linfty}, \eqref{linfyI}, the $L^\infty(\Omega)$ bound of $P$ established in Lemma \ref{uniform_bounded} with $4.$ of Lemma \ref{lem:local_existence} allow us to conclude that $T_{\max} = \infty$ .
\end{proof}

\subsection{The case $p>2, 0<\gamma<1$}

We now consider the case of the infection causing the prey to have a faster or enhanced mortality, and also to disperse slower. In this situation, we can no longer expect a classical solution \cite{antontsev2015evolution}. There can be several difficulties with system \eqref{eq:pde_modeld2}, and well posedness has to be demonstrated. Herein, the analysis is more involved than with FTE through a semi-linearity, as in \cite{parshad2021some, banerjee2025two}. We aim to show the global in time existence of a bounded weak solution in this section.

\subsubsection{Preliminaries and $L^{q}(\Omega)$ bounds}
We first define a weak solution,

\begin{definition}
\label{def:d1}
Let $2<p$, let $T \in (0,\infty]$ and $\Omega \subset \mathbb{R}^{n}, n=1,2$ be a bounded domain, with smooth boundary. A triple $(S,I,P)$ of non-negative functions defined on $\Omega \times (0,T)$ is called a weak solution to \eqref{eq:model1} on $[0,T]$ if,

(i) $(S,P) \in L^{2}_{loc}([0,T];W^{1,2}(\Omega)), I \in L^{2}_{loc}([0,T];L^{2}(\Omega))$

(ii) $|\nabla I|^{p-2} \nabla I \in L^{1}_{loc}([0,T];L^{1}(\Omega));$

(iii) For every $\phi \in C^{\infty}_{0}(\Omega \times [0,T))$ we have

\begin{eqnarray}
&& -\int^{T}_{0}\int_{\Omega} S \phi_{t} dxdt - \int_{\Omega} S_{0}(x) \phi(x,0)dx \nonumber \\
&& = -\delta_{1}\int^{T}_{0}\int_{\Omega} \nabla S \cdot \nabla \phi dxdt +
 \int^{T}_{0}\int_{\Omega} r S \left(1-\frac{S+I}{K} -\lambda I\right) \phi dxdt, \nonumber \\
&& -\int^{T}_{0}\int_{\Omega} I \phi_{t} dxdt - \int_{\Omega} I_{0}(x) \phi(x,0)dx \nonumber \\
&& = 
- \delta_{2}\int^{T}_{0}\int_{\Omega} |\nabla I|^{p-2} \nabla I \cdot \nabla \phi dx dt + \int^{T}_{0}\int_{\Omega} (\lambda SI - \frac{mPI}{a+I} -\mu I^{\gamma}) \phi dxdt, \nonumber \\
&& -\int^{T}_{0}\int_{\Omega} P \phi_{t} dxdt - \int_{\Omega} P_{0}(x) \phi(x,0)dx \nonumber \\
&& = -\delta_{3}\int^{T}_{0}\int_{\Omega} \nabla P \cdot \nabla \phi dxdt 
+\chi_1 \int^{T}_{0}\int_{\Omega} \xi(P) P \nabla S \cdot \nabla \phi dxdt 
\nonumber \\
&&+ \chi_2 \int^{T}_{0}\int_{\Omega} \eta(P) P \nabla I \cdot \nabla \phi dxdt +\int^{T}_{0}\int_{\Omega}\left( \frac{mIP}{a+I} -dP\right) \phi dxdt.
\end{eqnarray}
\end{definition}

We next state some lemmas that will be frequently used.

\begin{lemma}
\label{lem:wlp}
Consider exponents $0 < p_0 < p_1 < \infty$ and a domain $\Omega$ that is a closed and bounded in $\mathbb{R}^{n}$, $n\geq 1$, and a $u(x) \in L^{p_1}(\Omega)$. Then,

\begin{equation}
||u||_{L^{p_{\theta}}(\Omega)} \leq C||u||^{1-\theta}_{L^{p_{0}}(\Omega)} ||u||^{\theta}_{L^{p_{1}}(\Omega)}
\end{equation}

for all $0 \leq \theta \leq 1$, and where $p_{\theta}$ is defined as $\frac{1}{p_{\theta}} = \frac{1-\theta}{p_0} + \frac{\theta}{p_1}$.
\end{lemma}

\begin{lemma}
\label{lem:gns}
Consider a $\phi \in W^{m,q^{'}}(\Omega) \cap L^{q}(\Omega)$. Then if $p^{'}, q^{'}, q \geq 1, 0 \leq \theta \leq 1$, and

\begin{equation}
k - \frac{n}{p^{'}} \leq \theta \left(  m - \frac{n}{q^{'}} \right) - (1 - \theta) \frac{n}{q},
\end{equation}
there exists a constant C such that

\begin{equation}
	||\phi||_{W^{k,p^{'}}(\Omega)} \leq C ||\phi||^{\theta}_{W^{m,q^{'}}(\Omega)} ||\phi||^{1 - \theta}_{L^{q}(\Omega)} 
	\end{equation}.

\end{lemma}	

Now we state the classical Aubin-Lions compactness Lemma, \cite{robinson2003infinite}.

\begin{lemma}
\label{lem:al}
Let $X_0, X$ and $X_1$ be three reflexive Banach spaces with 

$X_0 \hookrightarrow \hookrightarrow X \hookrightarrow X_1$. That is $ X_0$ is compactly embedded in $X$ and that $X$ is continuously embedded in $X_1$. For $1\leq p,q \leq \infty$, let

 \begin{equation}
W = \{ u \in L^{p^{'}}([0,T];X_0) \ | \  u^{'} \in L^{q^{'}}([0,T];X_1)  \}
\end{equation}

(i) If $p^{'} < \infty$, then the embedding of $W$ into $L^{p^{'}}([0,T];X)$ is compact.

(ii) If $p^{'} = \infty$ and $q^{'} > 1$, then the embedding of $W$ into $C([0,T];X)$ is compact.
\end{lemma}

\begin{lemma}
\label{lem:pc1}
    Consider functions $\phi, \sigma \in C^{1}(\Omega), \Omega \subset \mathbb{R}^{n}, n=1,2, |\Omega| < \infty$, $\epsilon > 0$. Then if $p \geq 2$, we have
    \begin{equation}
        (|\nabla \phi|^{2} + \epsilon)^{\frac{p-2}{2}}|\nabla \phi|^{2} \geq |\nabla \phi|^{p},
\end{equation}

\begin{equation}
        (|\nabla \phi|^{2} + \epsilon)^{\frac{p-2}{2}}|\nabla \phi| \leq C\left(|\nabla \phi|^{p}+1\right),
\end{equation}

and 

\begin{equation}
        (|\phi|^{p-2}\phi - |\sigma|^{p-2}\sigma)(\phi - \sigma) \geq C |\phi - \sigma|^{p}.
\end{equation}

\end{lemma}

 To tackle well posedness, we first posit the regularized system

\begin{equation}
\label{eq:pde_modeld2}
\left\{ \begin{array}{ll}
S^{\epsilon}_t &= \delta_1 \nabla \cdot (\nabla S^{\epsilon})  +  rS^{\epsilon}\Big(1-\frac{(S^{\epsilon}+I^{\epsilon})}{K}\Big) - \lambda S^{\epsilon} I^{\epsilon}, \quad x\in \Omega, t>0\\
		I^{\epsilon}_t &= \nabla \cdot a(x,I^{\epsilon},\nabla I^{\epsilon})   + \lambda S^{\epsilon}I^{\epsilon} - \frac{mP^{\epsilon}I^{\epsilon}}{a+I^{\epsilon}} -\mu (I^{\epsilon})^{\gamma}, \quad x\in\Omega,t>0, \\
	 a(x,I^{\epsilon},\nabla I^{\epsilon}) & = \Big( \delta_{2}  \left(|\nabla I^{\epsilon}|^{2}+{\epsilon}\right)^{\frac{p-2}{2}}  \nabla I^{\epsilon} \Big) , \ p > 2, \ 0 < \gamma < 1 \\
 P^{\epsilon}_t &= \delta_3 \nabla \cdot (\nabla P^{\epsilon}) - \chi_1 \nabla \cdot (\xi(P^{\epsilon}) P^{\epsilon} \nabla S^{\epsilon}) - \chi_2\nabla \cdot(\eta(P^{\epsilon}) P^{\epsilon} \nabla I^{\epsilon}) + \dfrac{m I^{\epsilon} P^{\epsilon}}{a + I^{\epsilon}} - dP^{\epsilon} \\ 
\nabla S^{\epsilon} \cdot \eta &= |\nabla I^{\epsilon}|^{2}\nabla I^{\epsilon} \cdot \eta = \nabla P^{\epsilon} \cdot \eta =0, \quad x\in \partial \Omega \\
		S^{\epsilon}(x,0)&=S^{\epsilon}_0(x), P^{\epsilon}(x,0)=P^{\epsilon}_{0} ,I^{\epsilon}(x,0)=I^{\epsilon}_{0} (x), \quad x\in \Omega \\
\end{array}\right.
\end{equation}

Here $\Omega \subset \mathbb{R}^n$, with smooth boundary  $n=1,2, \ |\Omega| < \infty$. The local existence of the above is standard and we can state the following lemma.

\begin{lemma}
\label{main_result}
  Consider system \eqref{eq:pde_modeld2}. If the initial conditions are such that $(S^{\epsilon}_0, I^{\epsilon}_0, P^{\epsilon}_0) \in (W^{1, p}(\Omega))^3$ with $S^{\epsilon}_0, I^{\epsilon}_0, P^{\epsilon}_0 \geq 0$ and $p > n$ holds true and Assumption \ref{assu2} holds true.  Then there exists a classical solution  to \eqref{eq:pde_modeld2}, with $(S^{\epsilon}, I^{\epsilon}, P^{\epsilon}) \in (C([0, \infty); W^{1, p}(\Omega)) \cap (C^{2, 1}(\bar{\Omega} \times (0, T_{max, \epsilon})))^3$. If $T_{max, \epsilon} < \infty$, then $\lim_{t \rightarrow T_{max, \epsilon}} ||S^{\epsilon}||_{W^{1, \infty}(\Omega)} + ||I^{\epsilon}||_{L^{\infty}(\Omega)} + ||P^{\epsilon}||_{L^{\infty}(\Omega)} = + \infty$.
\end{lemma}

We begin to make uniform estimates so as to prove the existence of a bounded global weak solution to \eqref{eq:pde_modeld2}. We state the following lemma, about the regularized system,

\begin{lemma}
\label{lem:lq1}
Consider system \eqref{eq:pde_modeld2}. If the initial data $(S_{0}(x), P_{0}(x)) \in W^{1,\infty}(\Omega), I_{0}(x) \in L^{\infty}(\Omega)$, then if Assumption \ref{assu2} holds true, for any set of positive problem parameters and $q \geq 1$, there exists a $C=C(q)$ s.t. for any $1>>\epsilon > 0$, we have

\begin{equation}
    ||I^{\epsilon}||_{L^{q}(\Omega)} \leq C,
\end{equation}
  for all $t \in [0, T_{max, \epsilon})$  .
\end{lemma}

\begin{proof}
We multiply the equation for $I^{\epsilon}$ in \eqref{eq:pde_modeld2} by $q(I^{\epsilon})^{q-1}$, to obtain

\begin{eqnarray}
&& \frac{d}{dt}||I^{\epsilon}||^{q}_{q} 
 + \delta_{2}  q(q-1)\int_{\Omega} \left(
|\nabla I^{\epsilon}|^{2}+{\epsilon}\right)^{\frac{p-2}{2}}|\nabla I^{\epsilon}|^{2} (I^{\epsilon})^{q-2} dx
+q||I^{\epsilon}||^{q -1 +\gamma}_{q - 1 +\gamma} \nonumber \\
&& + mq\int_{\Omega} P^{\epsilon} \frac{(I^{\epsilon})^{q}}{a+I^{\epsilon}}dx =\lambda q\int_{\Omega}S^{\epsilon} (I^{\epsilon})^{q}dx \nonumber .\\
\end{eqnarray}
Note that when $2<p$, we have
\begin{equation}
\int_{\Omega} \left(
|\nabla I^{\epsilon}|^{2}+{\epsilon}\right)^{\frac{p-2}{2}}|\nabla I^{\epsilon}|^{2} (I^{\epsilon})^{q-2} dx
\geq 
\int_{\Omega} 
|\nabla I^{\epsilon}|^{p}(I^{\epsilon})^{q-2} dx . 
\end{equation}

Using the above estimate and positivity of $P_{\epsilon}$ entails,

\begin{equation}
\label{eq:est1}
 \frac{d}{dt}||I^{\epsilon}||^{q}_{q} 
 + \delta_{2}  q(q-1)\int_{\Omega} \left(
|\nabla I^{\epsilon}|\right)^{p}(I^{\epsilon})^{q-2} dx
+q||I^{\epsilon}||^{q - 1 +\gamma}_{q - 1 +\gamma}  \\
  \leq q\lambda\int_{\Omega}S^{\epsilon} (I^{\epsilon})^{q}dx. \nonumber \\
\end{equation}

Note that
\begin{equation}
\label{eq:gtl}
\int_{\Omega} (I^{\epsilon})^{q-2}|\nabla I^{\epsilon}|^{p}dx = \frac{p^{p}}{(p+q-2)^{p}}||\nabla (I^{\epsilon})^{\frac{p+q-2}{p}}||^{p}_{p}.
\end{equation}

Using the above, the supremum estimate on $S_{\epsilon}$ via Theorem \ref{thm:main_result}, lemma \ref{lem:local_existence}, and adding $||I^{\epsilon}||^{q}_{q}$ to both sides of \eqref{eq:est1} we obtain

\begin{equation}
\label{eq:iil}
 \frac{d}{dt}||I^{\epsilon}||^{q}_{q} +  ||I^{\epsilon}||^{q}_{q}+
\frac{p^{p}(\delta_{2}  q(q-1))}{(p+q-2)^{p}}||\nabla (I^{\epsilon})^{\frac{p+q-2}{p}}||^{p}_{p} + q||I^{\epsilon}||^{q - 1 +\gamma}_{q - 1 + \gamma}
 \leq (q\lambda K +1)||I^{\epsilon}||^{q}_{q}. 
\end{equation}

We have that

\begin{equation}
\int_{\Omega} (I^{\epsilon})^{q}dx
=\int_{\Omega}\left ((I^{\epsilon})^{\frac{p+q-2}{p}} \right)^{\frac{q p }{p+q-2}}dx.
\end{equation}

Now, we use GNS inequality on the function $\phi = (I)^{\frac{p+q-2}{p}}$ in Lemma \ref{lem:gns}. Thus

\begin{equation}
||\phi||_{L^{m}(\Omega)} \leq ||\nabla \phi||^{\theta}_{L^{r}(\Omega)} ||\phi||^{1-\theta}_{L^{q^{*}}(\Omega)}, 
\end{equation}
with

\begin{equation}
    m=\frac{q p }{p+q-2}, \ r = p, \ n=2.
\end{equation}

We need a choice of a $q^{*}$, such that $q<q^{*}$, and $\theta <1$, such that

\begin{equation}
 1 > \theta = \frac{\frac{1}{q^{*}} - \frac{p+q-2}{p q }}{\frac{1}{q^{*}} - \frac{1}{p} + \frac{1}{2}}.
\end{equation}
This is always possible, since $p>2$ we have  $\frac{p-2}{2} + \frac{p+q-2}{q} > 0$. It turns out that 

\begin{equation}
q^{*} = \frac{(1-\theta)2pq}{2(p+q-2)+\theta(p-2)q}.
\end{equation}
 Thus, we can choose $0< \theta <1$, such that $q^{*} = \frac{(1-\theta)2pq}{2(p+q-2)+\theta(p-2)q} \leq q-1$. By using Holder-Young with $\epsilon_{1}$, with exponents $p^{'},q^{'}$ where $p^{'}=\frac{p}{\theta}, q^{'}=\frac{q^{*}}{1-\theta}$, and further $0<\theta<1$ is chosen to ensure that $\frac{\theta}{p} + \frac{1-\theta}{q^{*}} = 1$, yields

\begin{eqnarray}
\label{eq:112}
&& \frac{d}{dt}||I^{\epsilon}||^{q}_{q} +  ||I^{\epsilon}||^{q}_{q} +
\frac{p^{p}(\delta_{2}q(q-1))}{(p+q-2)^{p}}||\nabla (I^{\epsilon})^{\frac{p+q-2}{p}}||^{p}_{p} \nonumber \\
 && \leq \epsilon_{1}^{\frac{p}{\theta}}   ||\nabla (I^{\epsilon})^{\frac{p+q-2}{p}}||^{p}_{p} + f(\epsilon_{1})C(q\lambda K + 1)   ||I^{\epsilon}||^{q^{*}}_{q^{*}}. 
\end{eqnarray}

Here $f(\epsilon)$ may depend on other parameters, and $C$ is the constant from the GNS inequality. By a choice of $\epsilon_{1} \leq \left(\frac{p^{p}(\delta_{2}q(q-1))}{(p+q-2)^{p}} \right)^{\frac{\theta}{p}}$, we can use the embedding of $L^{q}(\Omega) \hookrightarrow L^{q-1}(\Omega) \hookrightarrow L^{q^{*}}(\Omega)$, to yield

\begin{equation}
\label{eq:11}
 \frac{d}{dt}||I^{\epsilon}||^{q}_{q} +   C_{1}||I^{\epsilon}||^{q}_{q} \leq
C_{2}\left(||I^{\epsilon}||^{q}_{q}\right)^{\frac{q^{*}}{q}}.
\end{equation}

The $L^{q}(\Omega)$ bound on $I^{\epsilon}$, follows, via comparison to the ODE, $y^{'}=-C_{1}y + C_{2}y^{\frac{q^{*}}{q}}$, for $C_{1}, C_{2}, 1>\frac{q^{*}}{q} > 0$, this gives us a bound. To make the bound explicit in \eqref{eq:112} we can use Young with $\epsilon_{1}$ to yield,

\begin{eqnarray}
&& \frac{d}{dt}||I^{\epsilon}||^{q}_{q} + C_{1}||I^{\epsilon}||^{q}_{q}\nonumber \\
&& \leq \frac{C_{1}}{2}||I^{\epsilon}||^{q}_{q} + C_{3}|\Omega|,
\end{eqnarray}
where $C_{3}=\left(C f(\epsilon) (q \lambda K +1) \right)^{\frac{q}{q-q^{*}}}$
which yields,

\begin{equation}
 \frac{d}{dt}||I^{\epsilon}||^{q}_{q} + \frac{C_{1}}{2}||I^{\epsilon}||^{q}_{q} \leq  C_{3}|\Omega|.
\end{equation}

Via Gronwall's lemma it follows that,

\begin{equation}
||I^{\epsilon}||^{q}_{q} \leq (1-e^{-\frac{C_{1}}{2}t}) \left(\frac{2C_{3}|\Omega|}{C_{1}}\right) + e^{-\frac{C_{1}}{2}t}||I^{\epsilon}(0)||^{q}_{q} \leq  \left(\frac{2C_{3}|\Omega|}{C_{1}}\right) + e^{-\frac{C_{1}}{2}t}||I^{\epsilon}(0)||^{q}_{q}.
\end{equation}
This proves the lemma.
\end{proof}

\subsubsection{Alikakos-Moser iteration}
In this section we use the previous bounds, and devise an iteration scheme to derive the supremum bounds on the regularised solution to \eqref{eq:pde_modeld2}, from the $L^{q}(\Omega)$ bounds via lemma \ref{lem:lq1}

\begin{lemma}
\label{lem:mai}
Consider the system  \eqref{eq:pde_modeld2}. If the initial data $S_{0}(x), P_{0}(x) \in W^{1,\infty}(\Omega), I_{0}(x) \in L^{\infty}(\Omega)$, and Assumption \ref{assu2} holds true, then for any  $q \geq 1$,there exists a $C > 0$ such that for any $1>>\epsilon > 0$, we have,

\begin{equation}
   ||S^{\epsilon}||_{W^{1,\infty}(\Omega)} \leq C, \ ||I^{\epsilon}||_{L^{\infty}(\Omega)} \leq C, \
   ||P^{\epsilon}||_{L^{\infty}(\Omega)} \leq C,
\end{equation}
  for all $t \in [0, T_{max, \epsilon})$ .
\end{lemma}

\begin{proof}
Here we multiply the $I^{\epsilon}$ equation in the case $k=1$, by $(I^{\epsilon})^{q_{k} - 1}$, and integrate by parts to yield,

\begin{eqnarray}
&& \frac{d}{dt}||I^{\epsilon}||^{q}_{q}  
 + \delta_{2} (q-1)\int_{\Omega} \left(
|\nabla I^{\epsilon}|^{2}+{\epsilon}\right)^{\frac{p-2}{2}}|\nabla I^{\epsilon}|^{2} (I^{\epsilon})^{q-2} dx
+||I^{\epsilon}||^{q-1+\gamma}_{q-1+\gamma} \nonumber \\
&& + m\int_{\Omega}  \frac{P^{\epsilon}(I^{\epsilon})^{q}}{a+I^{\epsilon}}dx =   \lambda \int_{\Omega} S^{\epsilon}|I^{\epsilon}|^{q}dx.\nonumber \\
\end{eqnarray}

Using \eqref{eq:gtl} and the estimates in lemma \ref{lem:lq1}, we obtain

\begin{eqnarray}
&& \frac{d}{dt}||I^{\epsilon}||^{q_{k}}_{q_{k}}  
 + \delta_{2} (q-1)\int_{\Omega} 
\lvert \nabla \left(I^{\epsilon}\right)^{\frac{p + q_{k} - 2}{p}} \rvert ^{p}dx \nonumber \\
&&+\mu ||I^{\epsilon}||^{q_{k}-1+\gamma}_{q_{k}-1+\gamma}
 + m\int_{\Omega}   \frac{P^{\epsilon}(I^{\epsilon})^{q_{k}}}{a+I^{\epsilon}}dx  \nonumber \\
 && =\lambda \int_{\Omega} S^{\epsilon}|I^{\epsilon}|^{q_{k}}dx.\nonumber \\
\end{eqnarray}

Now we repeat the methods of lemma \ref{lem:lq1}, to obtain

\begin{equation}
 \frac{d}{dt}||I^{\epsilon}||^{q_{k}}_{q_{k}} +  
\frac{\delta_{2} (q_{k}-1)p^{p}}{(p+q_{k}-2)^{p}}||\nabla (I^{\epsilon})^{\frac{p+q_{k}-2}{p}}||^{p}_{p} + \mu||I^{\epsilon}||^{q_{k} - 1 +\gamma}_{q_{k} - 1 + \gamma}
 \leq \lambda K||I^{\epsilon}||^{q_{k}}_{q_{k}},
\end{equation}

adding $||I^{\epsilon}||^{q_{k}}_{q_{k}}$ to both sides of the above, we obtain

\begin{equation}
 \frac{d}{dt}||I^{\epsilon}||^{q_{k}}_{q_{k}} + ||I^{\epsilon}||^{q_{k}}_{q_{k}}+  
\frac{\delta_{2} (q_{k}-1)p^{p}}{(p+q_{k}-2)^{p}}||\nabla (I^{\epsilon})^{\frac{p+q_{k}-2}{p}}||^{p}_{p} + \mu||I^{\epsilon}||^{q_{k} - 1 +\gamma}_{q_{k} - 1 + \gamma}
 \leq (1+\lambda K)||I^{\epsilon}||^{q_{k}}_{q_{k}}.
\end{equation}

Note that
\begin{equation}
\int_{\Omega} (I^{\epsilon})^{q_{k}}dx
=\int_{\Omega}\left ((I^{\epsilon})^{\frac{p+q_{k}-2}{p}} \right)^{\frac{q_{k} p }{p+q_{k}-2}}dx.
\end{equation}

Now we use GNS inequality on the function $\phi = (I)^{\frac{p+q_{k}-2}{p}}$ in lemma \ref{lem:gns}  with,

\begin{equation}
    m=\frac{q_{k} p }{p+q_{k}-2}, \ r = p, \ n=2.
\end{equation}

We need a choice of a $q^{*}$, and $\theta <1$, such that

\begin{equation}
 1 > \theta = \frac{\frac{1}{q^{*}} - \frac{p+q_{k}-2}{p q_{k} }}{\frac{1}{q^{*}} - \frac{1}{p} + \frac{1}{2}}.
\end{equation}
This is always possible since $p>2$, by applying GNS

\begin{equation}
  \int_{\Omega}\left ((I^{\epsilon})^{\frac{p + q_{k} - 2}{p}} \right)^{\frac{q_{k} p }{p + q_{k} - 2}}dx \leq \left (\lvert \lvert\nabla (I^{\epsilon})^{\frac{p + q_{k} - 2}{p}}\rvert \rvert^{p}_{p} \right)^{\theta} \left( \lvert \lvert (I^{\epsilon})^{\frac{p + q_{k} - 2}{p}}\rvert \rvert^{q^{*}}_{q^{*}} \right)^{1-\theta},
\end{equation}

with a choice of $q^{*} = \frac{(q_{k}-1) p }{p+q_{k}-2} > 0$. 
Thus by using Holder-Young with $\epsilon_{1}$, with exponents $p^{'},q^{'}$ where $p^{'}=\frac{p}{\theta}, q^{'}=\frac{q^{*}}{1-\theta}$, and $\epsilon_{1}= \frac{\delta_{2}(q_{k}-1)p^{p}}{(p+q_{k}-2)^{p}}$, yields

\begin{eqnarray}
&& \frac{d}{dt}||I^{\epsilon}||^{q_{k}}_{q_{k}} + ||I^{\epsilon}||^{q_{k}}_{q_{k}}+ \mu||I^{\epsilon}||^{q_{k} - 1 +\gamma}_{q_{k} - 1 + \gamma} +
\frac{\delta_{2}(q_{k}-1)p^{p}}{(p+q_{k}-2)^{p}}||\nabla (I^{\epsilon})^{\frac{p+q_{k}-2}{p}}||^{p}_{p} \nonumber \\
 && \leq \frac{\delta_{2}(q_{k}-1)p^{p}}{(p+q_{k}-2)^{p}}  ||\nabla (I^{\epsilon})^{\frac{p+q_{k}-2}{p}}||^{p}_{p} + f(\epsilon_{1})(\lambda K+1)   ||I^{\epsilon}||^{q_{k}-1}_{q_{k}-1}. 
\end{eqnarray}

Next, we use Lemma \ref{lem:wlp} with $p_{0}=\frac{q_{k}-1}{2}, p_{1}=q_{k}-1+\frac{\gamma}{2}$, and a choice of $\theta$ such that $\frac{\theta}{\frac{q_{k}-1}{2}} + \frac{1-\theta}{q_{k}-1+\frac{\gamma}{2}} = \frac{1}{q_{k} - 1}$, to obtain

\begin{equation}
    ||I^{\epsilon}||^{q_{k}-1}_{q_{k}-1} \leq C ||I^{\epsilon}||^{(\theta)(q_{k}-1)}_{q_{k}-1+\frac{\gamma}{2}} 
    ||I^{\epsilon}||^{(1-\theta)(q_{k}-1)}_{\frac{q_{k}-1}{2}} = \left(||I^{\epsilon}||^{q_{k}-1+\frac{\gamma}{2}}_{q_{k}-1+\frac{\gamma}{2}} \right)^{\frac{(\theta)(q_{k}-1)}{q_{k}-1+\frac{\gamma}{2}}} \left(||I^{\epsilon} ||^{\frac{q_{k}-1}{2}}_{\frac{q_{k}-1}{2}} \right)^{2(1-\theta)}.
\end{equation}

Here the $C=\frac{C}{f(\epsilon) (\lambda K+1)}$ above. Also the inequality above follows via Holder-Young with $\epsilon_{1}$ with
$p=\frac{q_{k}-1+\frac{\gamma}{2}}{(\theta)(q_{k}-1)}$, $q=\frac{p}{p-1}$, yields

\begin{equation}
||I^{\epsilon}||^{q_{k}-1}_{q_{k}-1} \leq C||I^{\epsilon}||^{q_{k}-1+\frac{\gamma}{2}}_{q_{k}-1+\frac{\gamma}{2}} + \left(||I^{\epsilon} ||^{\frac{q_{k}-1}{2}}_{\frac{q_{k}-1}{2}} \right)^{\frac{\left( \frac{(2(1-\theta))(q_{k}-1+\frac{\gamma}{2})}{(\theta)(q_{k}-1)}\right)}{ \left(\frac{q_{k}-1+\frac{\gamma}{2}}{(\theta)(q_{k}-1)} -1  \right)}}.
\end{equation}

Lets call $p^{**} = \frac{\left( \frac{(2(1-\theta))(q_{k}-1+\frac{\gamma}{2})}{(\theta)(q_{k}-1)}\right)}{ \left(\frac{q_{k}-1+\frac{\gamma}{2}}{(\theta)(q_{k}-1)} -1  \right)} < \frac{\left( \frac{(2)(q_{k}-1+\frac{\gamma}{2})}{(\theta)(q_{k}-1)}\right)}{ \left(\frac{q_{k}-1+\frac{\gamma}{2}}{(\theta)(q_{k}-1)} -1  \right)}$.

Now via the standard embedding of $L^{q-1+\gamma}(\Omega)   \hookrightarrow L^{q - 1 + \frac{\gamma}{2}}(\Omega)   \hookrightarrow L^{q-1}(\Omega)$, and by the prescribed choice of $p^{**}$, since $p > 2 $, we ultimately have

\begin{equation}
\label{eq:1q1}
 \frac{d}{dt}||I^{\epsilon}||^{q_{k}}_{q_{k}} + ||I^{\epsilon}||^{q_{k}}_{q_{k}}+
+\mu||I^{\epsilon}||^{q_{k} - 1 +\gamma}_{q_{k} - 1 + \gamma} 
 \leq \frac{\mu}{2}\int_{\Omega} (I^{\epsilon})^{q_{k}-1+\gamma}dx  + C_{1}
+\left( ||I^{\epsilon}||^{\frac{q_{k}-1}{2}}_{\frac{q_{k}-1}{2}} \right)^{\left(\frac{p^{**}}{p^{**}-1}\right)}. 
\end{equation}

 Next we make a choice of $q_{k} = 2^{k}$, and we define


\begin{equation}
M_{K} = \max \left\{||I_{0}||_{\infty}, \sup_{t \in (0,T)} \int_{\Omega} |I^{\epsilon}|^{q_{k}} dx \right \} \ \mbox{for fixed} \ T \in (0,T_{max, \epsilon}).
\end{equation}

From here it is immediate that

\begin{equation}
M_{K} \leq ||I_{0}||^{q_{k}}_{q_{k}} + C_{2} (M_{K-1})^{\frac{f(K)}{f(K)-1}} + C_{1},
\end{equation}

where $f(k) = \frac{\left( \frac{(2)(2^{k}-1+\frac{\gamma}{2})}{(\theta)(2^{k}-1)}\right)}{ \left(\frac{2^{k}-1+\frac{\gamma}{2}}{(\theta)(2^{k}-1)} -1  \right)} $
and so via a recursion, and taking the limit as $k \rightarrow \infty$, yields,

\begin{eqnarray}
&&||I^{\epsilon}||_{\infty}  \nonumber \\
&\leq& \limsup_{k \rightarrow \infty} (M_{K})^{\frac{1}{q_{k}}} \nonumber \\
&\leq& 
(||I_{0}||^{q_{k}}_{q_{k}} + C_{2} (M_{K-1})^{\frac{f(K)}{f(K)-1}} + C_{1})^{\frac{1}{q_{k}}}  \nonumber \\
&\leq &
||I_{0}||_{\infty} + C_{1}, \nonumber \\
\end{eqnarray}

or recursively, $(M_{K})^{\frac{1}{q_{k}}} \leq (M_{0})^{2^{k}}$. The $L^{\infty}(\Omega)$ bound on $I^{\epsilon}$ follows. For the $S_{\epsilon}$ component, via simple comparison, we have

\begin{equation}
    S^{\epsilon}_t \leq \delta_1 \Delta S^{\epsilon}  +  rS^{\epsilon}\Big(1-\frac{(S^{\epsilon})}{K}\Big).
\end{equation}

Thus via standard theory of the logistic equation, the $W^{1,\infty}(\Omega)$ bound on $S_{\epsilon}$ follows. The $L^{\infty}(\Omega)$ bound on $P^{\epsilon}$ follows via applying the same methodology of Theorem \ref{thm:main_result} to the regularised system \eqref{eq:pde_modeld2}. 
Thus the lemma is proved.

\end{proof}

\begin{remark}
    In the above, we assume,
    
    \begin{equation}
        \left(M_{K-1}\right)^{^{\left(\frac{p^{**}}{p^{**}-1}\right)} } \leq \left(M_{K-1}\right)^{^{\left(\frac{f(K)}{f(K)-1}\right)}},
    \end{equation}
    under the assumption that,
    $||I^{\epsilon}||^{\frac{q_{k-1}}{2}}_{\frac{q_{k-1}}{2}} \geq 1$, else it is already bounded to begin with, and there is nothing to prove. Even if we are in this case the estimates in \eqref{eq:1q1} still hold, with a $C$ in front of the $||I^{\epsilon}||^{\frac{q_{k-1}}{2}}_{\frac{q_{k-1}}{2}} $ term, to scale appropriately. 
\end{remark}

\subsection{Global Existence of Weak Solutions}

In this section we prove global existence of weak solutions to \eqref{eq:pde_modeld2}. Herein we have to derive estimates uniform in the parameter $\epsilon$.

\begin{lemma}
\label{lem:es11}
    Consider the system \eqref{eq:pde_modeld2}. If the initial data $S_{0}(x), P_{0}(x) \in W^{1,\infty}(\Omega), I_{0}(x) \in L^{\infty}(\Omega)$, then for any $1>>\epsilon > 0$, there exists a constant $C$, depending on the problem parameters, but independent of $\epsilon$, such that we have

\begin{equation}
\label{eq:l11n}
    ||S^{\epsilon}||_{W^{1,\infty}(\Omega)} \leq C,
\end{equation}

\begin{equation}
\label{eq:l12n}
    ||I^{\epsilon}||_{L^{\infty}(\Omega)} \leq C,
\end{equation}

\begin{equation}
\label{eq:l13n}
    ||P^{\epsilon}||_{L^{\infty}(\Omega)} \leq C,
\end{equation}

\begin{equation}
\label{eq:l14n}
\int^{t}_{0}\int_{\Omega} 
|\nabla I^{\epsilon}(.,s)|^{p}(I^{\epsilon}(.,s))^{q-2} dx ds \leq C, \ p,q >2,
\end{equation}

\begin{equation}
\label{eq:l15n}
\int^{t}_{0}\int_{\Omega} 
|\nabla I^{\epsilon}(.,s)|^{p} dx ds \leq C.
\end{equation}
\end{lemma}

\begin{proof}
Estimates \eqref{eq:l11n}-\eqref{eq:l13n} follow directly from lemma \ref{lem:mai}. Integration of \eqref{eq:iil} in time from $[0,t]$, the embedding of $L^{\infty}(\Omega) \hookrightarrow L^{q}(\Omega)$ and the direct use of lemma \ref{lem:mai} again yields \eqref{eq:l14n}. Setting $q=2$ therein yields \eqref{eq:l15n}.
\end{proof}

\begin{theorem}
\label{thm:es22}
    Consider the system \eqref{eq:pde_modeld2}. If the initial data $(S_{0}(x), P_{0}(x)) \in W^{1,\infty}(\Omega), I_{0}(x) \in L^{\infty}(\Omega)$, then there exists a function $I \in L^{p}_{loc}((0,\infty), W^{1,p}(\Omega)) \cap L^{\infty}((0,\infty);L^{\infty}(\Omega))$ and $(S,P) \in  L^{2}_{loc}((0,\infty), W^{1,2}(\Omega)) \cap L^{\infty}((0,\infty);L^{\infty}(\Omega))$  as well as a $\Gamma \in L^{\frac{p}{p-1}}_{loc}((0,\infty);L^{\frac{p}{p-1}}(\Omega)) \cap L^{\infty}((0,\infty);L^{\infty}(\Omega))$, and a sequence of approximants $\epsilon = \epsilon_{j} \searrow 0$, such that

\begin{equation}
 S^{\epsilon} \rightharpoonup  S^{*}  \ \mbox{in} \ L^{2}_{loc}((0,\infty);W^{1,2}(\Omega))
\end{equation}

    \begin{equation}
    S^{\epsilon} \overset{*}{\rightharpoonup} S^{*} \ \mbox{in} \ L^{\infty}((0,\infty);L^{\infty}(\Omega))
\end{equation}

\begin{equation}
\nabla S^{\epsilon} \overset{*}{\rightharpoonup} \nabla S^{*}  \ \mbox{in} \ L^{\infty}((0,\infty);L^{\infty}(\Omega))
\end{equation}

\begin{equation}
 P^{\epsilon} \rightharpoonup  P^{*}  \ \mbox{in} \ L^{2}_{loc}((0,\infty);W^{1,2}(\Omega))
\end{equation}

 \begin{equation}
P^{\epsilon} \overset{*}{\rightharpoonup} P^{*} \ \mbox{in} \ L^{\infty}((0,\infty);L^{\infty}(\Omega))
\end{equation}



 \begin{equation}
I^{\epsilon} \overset{*}{\rightharpoonup} I^{*}  \ \mbox{in} \ L^{\infty}((0,\infty);L^{\infty}(\Omega))
\end{equation}

 \begin{equation}
\nabla I^{\epsilon} \rightharpoonup \nabla I^{*}  \ \mbox{in} \ L^{p}_{loc}((0,\infty);L^{p}(\Omega))
\end{equation}

 \begin{equation}
 \label{eq:nlt}
|\nabla I^{\epsilon}|^{p-2}\nabla I^{\epsilon} \rightharpoonup  \Gamma \ \mbox{in} \ L^{\frac{p}{p-1}}_{loc}((0,\infty);L^{\frac{p}{p-1}}(\Omega)).
\end{equation}

\end{theorem}

\begin{proof}
    These estimates follow directly from the estimates in lemma \ref{lem:es11}. For the convergence of the nonlinear term, $\int^{t}_{0}\int_{\Omega} 
|\nabla I^{\epsilon}(.,s)|^{p-2}(\nabla I^{\epsilon}(.,s))|^{\frac{p}{p-1}} dx ds = \int^{t}_{0}\int_{\Omega} 
|\nabla I^{\epsilon}(.,s)|^{p}dxds \leq C$. Thus equation \eqref{eq:nlt} follows.
\end{proof}
Next, we state a result about the time derivative of the solution.

\begin{theorem}
\label{thm:td1}
    Consider the system \eqref{eq:pde_modeld2}. If the initial data $S_{0}(x),P_{0}(x) \in W^{1,\infty}(\Omega), I_{0}(x) \in L^{\infty}(\Omega)$, then there exists a constant $C$ such that for any time $T > 0$, we have

     \begin{equation} 
\left \lVert \frac{\partial I^{\epsilon}}{\partial t}\right \rVert_{L^{1}((0,T); (W^{1,p}(\Omega))^{*})} \leq C.
\end{equation}

\end{theorem}

\begin{proof}
We consider a test function $\zeta \in C^{\infty}_{0}(\Omega)$ such that $||\zeta||_{W^{1,p}(\Omega)} \leq 1$, note that

 \begin{equation} 
\left \lVert \frac{\partial I^{\epsilon}}{\partial t}\right \rVert_{L^{1}((0,T); (W^{1,p}(\Omega))^{*})} = \int^{T}_{0}\left( \sup_{\zeta \in C^{\infty}_{0}(\Omega),||\zeta||_{W^{1,p}(\Omega)} \leq 1 } \int_{\Omega} \frac{\partial I^{\epsilon}}{\partial t} \zeta dx\right) dt.
\end{equation}

We now consider

\begin{eqnarray}
    && \int_{\Omega} \frac{\partial I^{\epsilon}}{\partial t} \zeta dx \nonumber \\
    && = \ \int_{\Omega} [\nabla \cdot a(\nabla I^{\epsilon})  + \lambda S^{\epsilon}I^{\epsilon} - \frac{mP^{\epsilon}I^{\epsilon}}{a+I^{\epsilon}} -\mu (I^{\epsilon})^{\gamma}] \zeta dx \nonumber \\
    && = \int_{\Omega} [\nabla \cdot \Big(\delta_2   \left(|\nabla I^{\epsilon}|^{2}+{\epsilon}\right)^{\frac{p-2}{2}}  \nabla I^{\epsilon} \Big)   + \lambda S^{\epsilon}I^{\epsilon} - \frac{mP^{\epsilon}I^{\epsilon}}{a+I^{\epsilon}} -\mu (I^{\epsilon})^{\gamma} ]\zeta dx \nonumber \\
    && \leq  \lambda K \int_{\Omega} | \zeta| |I^{\epsilon}| dx + m\int_{\Omega} | \zeta| |P^{\epsilon}| dx + \mu\int_{\Omega} | \zeta| |(I^{\epsilon})^{\gamma}| dx  \nonumber \\
    && + \int_{\Omega}\Big(\delta_{2}   \left(|\nabla I^{\epsilon}|^{2}+{\epsilon}\right)^{\frac{p-2}{2}}  \nabla I^{\epsilon} \Big) \nabla \zeta dx \nonumber \\
    &&  \leq \lambda K ||\zeta||_{\infty}\int_{\Omega}  |I^{\epsilon}| dx + m||\zeta||_{\infty}\int_{\Omega} |P^{\epsilon}| dx + \mu ||\zeta||_{\infty} \int_{\Omega}  |I^{\epsilon}|^{\gamma} dx  \nonumber \\
    &&+ \delta_2 \int_{\Omega}  \left(|\nabla I^{\epsilon}|^{p-1}+1\right) |\nabla \zeta| dx \nonumber \\
    && \leq \lambda K ||\zeta||_{\infty}\int_{\Omega}  |I^{\epsilon}| dx + m||\zeta||_{\infty}\int_{\Omega} |P^{\epsilon}| dx + \mu ||\zeta||_{\infty} \int_{\Omega}  |I^{\epsilon}|^{\gamma} dx  \nonumber \\
    &&+ C_{1} \int_{\Omega}  |\nabla I^{\epsilon}|^{p}  dx +  C_{2} \int_{\Omega}  |\nabla \zeta|^{p}  dx + C_{3}\nonumber \\
    && \leq C.
\end{eqnarray}

This follows via application of Holder-Young on the $ \int_{\Omega}  \left(|\nabla I^{\epsilon}|^{p-1}+1\right) |\nabla \zeta| dx$ term, and  uniform estimates in Lemma \ref{lem:es11}.
\end{proof}

\begin{lemma}
\label{lem:ala1}
Consider the system \eqref{eq:pde_modeld2}, with $ 2 < p $, there exists a subsequence ${\epsilon_{j}} \rightarrow \epsilon$ as $j \rightarrow \infty$, such that
 \begin{equation}
I^{\epsilon_{j}} \rightarrow I^{*} \ \mbox{in} \  L^{p}([0,T];L^{p}(\Omega)) \ \mbox{as}, \ j \rightarrow \infty
\end{equation}
\end{lemma}

\begin{proof}
    We know that via estimates in Lemma \ref{lem:es11}, we have
    \begin{equation}
I^{\epsilon_{j}} \in L^{p}((0,T), W^{1,p}(\Omega)).
\end{equation}
\end{proof}
Furthermore via Theorem \ref{thm:td1}, we have 

 \begin{equation}
 \frac{\partial I^{\epsilon}}{\partial t} \in L^{1}((0,T); (W^{1,p}(\Omega))^{*}).
\end{equation}

We now use the classical Aubin-Lion compactness lemma \ref{lem:al}, with the function spaces

 \begin{equation}
W^{1,p}(\Omega) \hookrightarrow \hookrightarrow L^{p}(\Omega) \hookrightarrow (W^{1,p}(\Omega))^{*},
 \end{equation}

 to obtain the compactness of the subsequence $ I^{\epsilon_{j}}$ in $L^{p}((0,T), L^{p}(\Omega))$, thus yielding,
 \begin{equation}
     I^{\epsilon_{j}} \rightarrow I^{*} \ \mbox{in} \ L^{p}((0,T), L^{p}(\Omega)).
     \end{equation}

Subsequently, we have
\begin{equation}
\lim_{j \rightarrow \infty}  \int^{T}_{0}\int_{\Omega}|I^{\epsilon_{j}}-I^{*}|dxdt \rightarrow 0. 
 \end{equation}
Note as $\lim_{j \rightarrow \infty}\epsilon_{j} \rightarrow 0$
, this follows via 
\begin{equation}
L^{p}((0,T), L^{p}(\Omega)) \hookrightarrow L^{1}((0,T), L^{1}(\Omega)).
 \end{equation}

\subsubsection{Convergence of the Non-linear gradient term} 

\begin{lemma}
\label{lem:cnt}
Consider the system \eqref{eq:pde_modeld2} with $2<p$, there exists a subsequence ${\epsilon_{j}}$, such that for any $T>0$

\begin{equation}
|\nabla I^{\epsilon_{j}}|^{p-2}  \nabla I^{\epsilon_{j}} \rightarrow |\nabla I|^{p-2}  \nabla I \ \mbox{in} \ L^{\frac{p}{p-1}}(\Omega \times (0,T)).
 \end{equation}
\end{lemma}

We will first prove convergences in the regularized system \eqref{eq:pde_modeld2}. We first consider a weak formulation of solution to system \eqref{eq:pde_modeld2}, that is for a $\phi \in C^{\infty}_{0}(\overline{\Omega} \times [0,T))$, we have

\begin{eqnarray}
 && \lim_{j \rightarrow \infty} \int^{T}_{0} \int_{\Omega} \frac{\partial I^{\epsilon_{j}}}{\partial t} \phi dx dt \nonumber \\
 && = - \lim_{j \rightarrow \infty}\int^{T}_{0} \int_{\Omega} \frac{\partial \phi}{\partial t} I^{\epsilon_{j}} dx dt + \int_{\Omega}I^{\epsilon_{j}}_{0}(x)\phi_{0}(x)dx \nonumber \\
 && =\int^{T}_{0} \int_{\Omega} \frac{\partial \phi}{\partial t} I dx dt + \int_{\Omega}I_{0}(x)\phi_{0}(x)dx . \  \nonumber \\
\end{eqnarray}
This follows as $\nabla I^{\epsilon} \rightharpoonup \nabla I^{*}  \ \mbox{in} \ L^{p}_{loc}((0,\infty);L^{p}(\Omega))$, and $L^{p}((0,T);W^{1,p}(\Omega)) \hookrightarrow L^{1}((0,T);L^{p}(\Omega))$.

The other semi-linear term convergences follow similarly,

\begin{eqnarray}
 && \lim_{j \rightarrow \infty} \lambda\int^{T}_{0} \int_{\Omega}  | S^{\epsilon_{j}}I^{\epsilon_{j}} - S I |\phi dx dt \nonumber \\
 && =  \lim_{j \rightarrow \infty} \lambda\int^{T}_{0} \int_{\Omega} |S^{\epsilon_{j}}I^{\epsilon_{j}}  - I^{\epsilon_{j}}S + I^{\epsilon_{j}}S - S I|  \phi dx dt \nonumber \\
 && \leq  \lim_{j \rightarrow \infty} \lambda ||I^{\epsilon_{j}}||_{L^{\infty} (0,T;L^{\infty}(\Omega))} \int^{T}_{0} \int_{\Omega} |S^{\epsilon_{j}}-S| \phi dxdt  \nonumber \\
 && +  \lim_{j \rightarrow \infty} \lambda ||S||_{L^{\infty} (0,T;L^{\infty}(\Omega))} \int^{T}_{0} \int_{\Omega} |I^{\epsilon_{j}}-I| \phi dxdt \nonumber \\
 && \leq 0.
\end{eqnarray}

This follows again from the estimates in Lemma \ref{lem:es11}, Theorem \ref{thm:es22} and $L^{p}((0,T);W^{1,p}(\Omega)) \hookrightarrow L^{1}((0,T);L^{p}(\Omega))$. For convergence of the semi-linear term $\lim_{j \rightarrow \infty} \int^{T}_{0} \int_{\Omega}  m (I^{\epsilon_{j}})^{\gamma} \phi dx dt$, we appeal again to Lemma \ref{lem:es11}, Theorem \ref{thm:es22} and $W^{1,p}(\Omega) \hookrightarrow \hookrightarrow L^{2}(\Omega) \hookrightarrow L^{1+\gamma}(\Omega)$, for $n=1,2$. Thus we have

\begin{eqnarray}
    && - \int^{T}_{0}\int_{\Omega} I \frac{\phi }{\partial t}  dx  - \int_{\Omega}I_{0}(x)\phi_{0}(x)dx \nonumber \\
    && = - \int^{T}_{0}\int_{\Omega} \Gamma \cdot \nabla \phi dx dt   + \int^{T}_{0}\int_{\Omega}[\lambda SI - \frac{mPI}{a+I} -\mu (I)^{\gamma}]\phi dx dt. \nonumber \\
   \end{eqnarray}

At this juncture, we will have that $(S,I,P)$ solve system \eqref{eq:model1}, for $p>2, 0 < \gamma < 1$, in the weak sense defination \ref{def:d1}, as long as we can show convergence of the non-linear gradient term. We proceed to this next. We show this next,

\begin{equation}
    \lim_{j \rightarrow \infty} \int^{T}_{0}\int_{\Omega}  \Gamma^{\epsilon_{j}} \cdot \nabla \phi dx dt = \int^{T}_{0}\int_{\Omega}  \Gamma \cdot \nabla \phi dx dt = 
    \int^{T}_{0}\int_{\Omega} |\nabla I|^{p-2} \nabla I \cdot \nabla \phi dx dt.
\end{equation}

Thus, we will show that 
\begin{equation}
 \int^{T}_{0}\int_{\Omega} (|\nabla I|^{p-2} \nabla I - \Gamma ) \cdot \nabla \phi dx dt = 0.
\end{equation}
By showing the above is true both in $(\geq)$ and $(\leq)$ sense.

\begin{lemma}
\label{lem:cnt2i}
Consider a function $\phi \in L^{p}((0,T);W^{1,p}(\Omega))$, then for solutions to system \eqref{eq:pde_modeld2} we have,

\begin{equation}
\label{eq:1een}
     \int^{T}_{0}\int_{\Omega} ( \Gamma - |\nabla \phi|^{p-2} \nabla \phi  ) \cdot (\nabla I - \nabla \phi )dx dt \geq  0.
\end{equation}
    
\end{lemma}

\begin{proof}
We consider
\begin{eqnarray}
\label{eq:1id1}
&& \int^{T}_{0}\int_{\Omega} ( |\nabla I^{\epsilon}|^{p-2} \nabla I^{\epsilon}  - |\nabla \phi|^{p-2} \nabla \phi  ) \cdot (\nabla I - \nabla \phi )dx dt \nonumber \\
    && = \int^{T}_{0}\int_{\Omega} ( |\nabla I^{\epsilon}|^{p-2} \nabla I^{\epsilon} \cdot (\nabla I - \nabla I^{\epsilon})dxdt
    +\int^{T}_{0}\int_{\Omega}(|\nabla \phi|^{p-2} \nabla \phi  ) \cdot (\nabla I^{\epsilon} - \nabla I)dxdt \nonumber \\
    && + \int^{T}_{0}\int_{\Omega} ( |\nabla I^{\epsilon}|^{p-2} \nabla I^{\epsilon}  - |\nabla \phi|^{p-2} \nabla \phi  ) (\nabla I^{\epsilon} - \nabla \phi)dxdt\nonumber \\
    && = I_{1} + I_{2} + I_{3}. \nonumber \\
\end{eqnarray}
    We see $I_{2} \rightarrow 0$ as ($\epsilon_{j}$ relabeled $\epsilon$) $\epsilon \rightarrow 0$, via the estimates in lemma \ref{lem:es11} and Theorem \ref{thm:es22}. Furthermore, $I_{3} \geq 0$ via the preliminary lemma \ref{lem:pc1}. We next tackle the $I_{1}$ term. We multiply system \eqref{eq:pde_modeld2} by $I-I^{\epsilon}$ and integrate by parts to yield,

\begin{eqnarray}
&& \int^{T}_{0}\int_{\Omega} \left( \frac{\partial I^{\epsilon}}{\partial t}  \right) ( I -  I^{\epsilon} )dx dt \nonumber \\
&& = \int^{T}_{0}\int_{\Omega} ( |\nabla I^{\epsilon}|^{2}  + \epsilon )^{\frac{p-2}{2}} \nabla I^{\epsilon}  \cdot  \nabla (I -   I^{\epsilon}) dx dt \nonumber \\
&& 
+\int^{T}_{0}\int_{\Omega} \left( \lambda S^{\epsilon}I^{\epsilon} - \frac{mP^{\epsilon}I^{\epsilon}}{a+I^{\epsilon}} -\mu (I^{\epsilon})^{\gamma}\right) ( I -  I^{\epsilon} )dx dt. \nonumber \\
\end{eqnarray}

     Thus we have that as $\epsilon \searrow 0$,

\begin{eqnarray}
&& \lim_{\epsilon \rightarrow 0} \int^{T}_{0}\int_{\Omega} \left( \frac{\partial I^{\epsilon}}{\partial t}  \right) ( I -  I^{\epsilon} )dx dt \nonumber \\
&& = \lim_{\epsilon \rightarrow 0} \int^{T}_{0}\int_{\Omega} ( |\nabla I^{\epsilon}|^{2}  + \epsilon )^{\frac{p-2}{2}} \nabla I^{\epsilon}  \cdot \nabla (I -  I^{\epsilon})dx dt \nonumber \\
&& + \lim_{\epsilon \rightarrow 0} \int^{T}_{0}\int_{\Omega} 
\left( \lambda S^{\epsilon}I^{\epsilon} - \frac{mP^{\epsilon}I^{\epsilon}}{a+I^{\epsilon}} -\mu (I^{\epsilon})^{\gamma}\right) 
( I -  I^{\epsilon} )dx dt.\nonumber \\
\end{eqnarray}

 Using lemma \ref{lem:ala1} and Theorem \ref{thm:td1} we have the L.H.S approaching 0. Whereas via lemma \ref{lem:es11} and lemma \ref{lem:ala1}, we have the second group of terms on the right hand side approaching zero. Furthermore,

 \begin{equation}
\lim_{\epsilon \rightarrow 0} \int^{T}_{0}\int_{\Omega} ( |\nabla I^{\epsilon}|^{2}  + \epsilon )^{\frac{p-2}{2}} \nabla I^{\epsilon} \cdot  ( \nabla ( I -   I^{\epsilon} ))dx dt =  \int^{T}_{0}\int_{\Omega} ( |\nabla I^{\epsilon}|^{p-2} \nabla I^{\epsilon} \cdot (\nabla I - \nabla I^{\epsilon})dxdt.
    \end{equation}

 Thus

 \begin{equation}
0 =  \lim_{\epsilon \rightarrow 0} \int^{T}_{0}\int_{\Omega} ( |\nabla I^{\epsilon}|^{p-2} \nabla I^{\epsilon} \cdot (\nabla I - \nabla I^{\epsilon}) + 0 + 0,
    \end{equation}
    thus $I_{1} \rightarrow 0$, which we now inject back into \eqref{eq:1id1} to yield \eqref{eq:1een} is true. Thus, choosing $\psi \in L^{p}(0,T;W^{1,p}(\Omega))$ we have, for $\lambda \psi = u - \phi$, with $\lambda \in \mathbb{R}$. Thus, we have inserting this form into \eqref{eq:1een},

    \begin{equation}
\int^{T}_{0}\int_{\Omega}(\Gamma - |\nabla I - \lambda \nabla \psi|^{p-2} \nabla (I - \lambda  \psi))   dxdt.
    \end{equation}
    Letting $\lambda \rightarrow 0^{+}$ and $\lambda \rightarrow 0^{-}$, we obtain the desired inequality. This proves the lemma.
\end{proof}

We can now state the main result of the section,

\begin{theorem}
\label{thm:mt1}
Consider the system \eqref{eq:model1}, when $p>2$, $0<\gamma<1$. If the initial data $S_{0}(x), P_{0}(x) \in W^{1,\infty}(\Omega), I_{0}(x) \in L^{\infty}$, then for any set of positive parameters 

$(\delta_{1},\delta_{2},\delta_{3},\lambda, K, r, m,a,\gamma, \chi_{1},\chi_{2},d)$, there exists a global weak solution $(S,I,P)$ to system \eqref{eq:model1}, in the sense of definition \ref{def:d1}. Furthermore, there exists a constant $C$, independent of $\epsilon$ such that the regularized solutions $(S^{\epsilon},I^{\epsilon},P^{\epsilon})$ to system \eqref{eq:pde_modeld2} satisfy,

\begin{equation}
   ||S^{\epsilon}||_{W^{1,\infty}(\Omega)} \leq C, \ ||I^{\epsilon}||_{L^{\infty}(\Omega)} \leq C, \
   ||P^{\epsilon}||_{L^{\infty}(\Omega)} \leq C,
\end{equation}
  for all $t > 0$.
\end{theorem}

\begin{proof}
    This follows via Lemma \ref{lem:mai}, Lemma \ref{lem:es11}, Lemma \ref{lem:ala1}, Lemma \ref{lem:cnt},  Lemma \ref{lem:cnt2i}, Theorem \ref{thm:es22} and Theorem \ref{thm:td1}.
\end{proof}

\section{Stability analysis of the positive equilibrium ($p=2, \gamma =1$ case)}
This section of our paper is devoted to discuss the stability of the positive equilibrium solution for system $(\ref{eq:model1})$. 
It is easy to see that system  $(\ref{eq:model1})$ has a unique positive equilibrium solution $E^*=(S^*,I^*,P^*)$ if and only if 
\begin{equation}\label{eq:eqmcnd}
 (\mathbf{H_1}):   \lambda K > \mu ~~~and~~~ m > d + \dfrac{\lambda ad(r+\lambda K)}{r(\lambda K - \mu)},
\end{equation}
where
\begin{equation}
    S^*=K- \dfrac{ad(r+\lambda K)}{r(m-d)}, ~~~ I^*= \dfrac{ad}{m-d}, ~~~ P^*= \dfrac{(a+ I^*)(\lambda S^* - \mu)}{m}.
\end{equation}
It has obtained the global stability of the  positive equilibrium solution $E^*$ for corresponding ODE system and reaction diffusion system without taxis in \cite{chattopadhyay2001pelicans}. In this section, we focus on how the inclusion of taxis affect the stability and enhance the possibility of pattern formulations. Here, we assume $\chi_1 = 0 $.\\
The linearized problem of system ($\ref{eq:model1}$) at constant steady state $E^*= (S^*,I^*,P^*)$ can be expressed by
\begin{equation}
    \begin{pmatrix}
        \frac{\partial S}{\partial t}\\[5pt]
        \frac{\partial I}{\partial t}\\[5pt] 
        \frac{\partial P}{\partial t}\\
    \end{pmatrix}
 = D \begin{pmatrix}
     \Delta{S} \\
     \Delta{I} \\
     \Delta{P} \\
 \end{pmatrix} + J_{(S^*,I^*,P^*)}
 \begin{pmatrix}
     S\\
     I\\ 
     P\\
 \end{pmatrix},
\end{equation}
where 
\begin{equation}
    D = \begin{pmatrix}
        \delta_1 & 0 & 0 \\
        0 & \delta_2 & 0 \\
        0 & - \chi_2\eta(P^*) P^* & \delta_3 \\
       \end{pmatrix}, ~~
J_{(S^*,I^*,P^*)}=
          \begin{pmatrix}
          a_{11} & a_{12} & a_{13} \\
          a_{21} & a_{22} & a_{23} \\
          a_{31} & a_{32} & a_{33} \\
           
           \end{pmatrix}
\end{equation}
and 
\begin{equation*}
    a_{11} = - \dfrac{rS^*}{K}, ~~~ a_{12} = -\left( \dfrac{r}{K} + \lambda \right)S^*, ~~~ a_{13}=0, \\
\end{equation*}
\begin{equation}
   a_{21} = \lambda I^*, ~~~~~~a_{22}= \dfrac{mI^*P^*}{(a+I^*)^2}, ~~~~~~ a_{23}= -d, \\ 
\end{equation}
 \begin{equation*}
     a_{31}= 0,~~~~~~a_{32}= \dfrac{amP^*}{(a+I^*)^2}, ~~~~~~ a_{33}=0, \\
 \end{equation*}
 
\begin{lemma}\label{lemma2}
     Suppose $\delta_1=\delta_2=\delta_3=0$, and $\chi_1=\chi_2=0$ then, under the hypothesis $(\mathbf{H_1})$, constant steady state $E^*= (S^*, I^*, P^*)$ is locally asymptotically stable if and only if
     
      \begin{align}\label{eq:eqcnd2}
         (\mathbf{H_2}): ~~ &m> \text{max} [d+\lambda ad(r+\lambda K)/{r(\lambda K - \mu)}, d\lambda K/r],
          &\left(\dfrac{r}{K} + \lambda \right)S^*I^*\lambda > \dfrac{2rmS^*I^*P^*}{K(a+I^*)^2}.
         \end{align}
     \end{lemma}
 \begin{proof}\cite[Theorem 2]{chattopadhyay2001pelicans} we can see in details.
 \end{proof}
 The stability of $E^*=(S^*,I^*,P^*)$ is determined by the eigenvalue problem
\begin{equation*}
    (D\Delta + J) \begin{pmatrix}
        \phi_1 \\ \phi_2 \\ \phi_3
    \end{pmatrix} = \gamma \begin{pmatrix}
        \phi_1 \\ \phi_2 \\ \phi_3
    \end{pmatrix}
\end{equation*}
that is 
\begin{equation}
\label{egnv}
   \begin{cases}
       \delta_1 \Delta \phi_1 + a_{11} \phi_1 + a_{12} \phi_2 + a_{13} \phi_3 = \gamma \phi_1, ~~~~~~~~~~~~x \in \Omega \\
       \delta_2  \Delta \phi_2 + a_{21} \phi_1 + a_{22} \phi_2 + a_{23} \phi_3 = \gamma \phi_2,  ~~~~~~~~~~~~x \in \Omega \\
       \delta_3 \Delta \phi_3 - \chi_2\eta(P^*)P^* \Delta \phi_2 + a_{31} \phi_1 + a_{32} \phi_2 + a_{33} \phi_3= \gamma \phi_3, ~~x \in \Omega \\
      \dfrac{\partial \phi_1}{\partial n}=\dfrac{\partial \phi_1}{\partial n}=\dfrac{\partial \phi_1}{\partial n}=0, ~~~~~~~~~~~~~~~~~~~~~~~~~~~x \in \partial \Omega\\
 \end{cases}    
\end{equation} 
Then, $\gamma$ is an eigenvalue of $D\gamma + J$ if and only if $\gamma$ is an eigenvalue of the matrix $-(\frac{k\pi}{l})^2D + J$ for each $k \geq 0 $, where $(\frac{k\pi}{l})^2(k \in \mathbb{N}_0)$ are the eigenvalues of $-\Delta$ under the Neumann boundary condition, where $\Omega = (0, l)$. It is easy to see that $\gamma$ satisfies the following characteristics equation
\begin{equation*}
    \Delta_k = \gamma ^3 + \mathcal{A}_k \gamma ^2 + \mathcal{B}_k \gamma + \mathcal{C}_k = 0,
\end{equation*}

where
\begin{align} \label{eqsys1}
\nonumber \mathcal{A}_k = &(\delta_1 + \delta_2 + \delta_3)(\frac{k\pi}{l})^2 +\mathcal{A},   \\
   \nonumber \mathcal{B}_k = & (\delta_1 \delta_2 + \delta_2 \delta_3 + \delta_3 \delta_1)(\frac{k\pi}{l})^4 + \mu_n \left[(\delta_2 + \delta_3)\dfrac{rS^*}{K} - (\delta_1 + \delta_3)\dfrac{mI^*P^*}{(a+I^*)^2}\right]+ \mathcal{B} -\chi_2 \eta(P^*) P^*d, \\
    \mathcal{C}_k = & 
    w_1 (\frac{k\pi}{l})^6 + w_2(\frac{k\pi}{l})^4  + w_3 (\frac{k\pi}{l})^2 + \mathcal{C} -\chi_2\eta(P^*)P^*\left(\delta_3 (\frac{k\pi}{l})^2 + \dfrac{rS^*}{K}d\right) , \\ 
   \nonumber \mathcal{P}_k= & \mathcal{A}_k \mathcal{B}_k-\mathcal{C}_k \\
   \nonumber=&m_1 (\frac{k\pi}{l})^6 + m_2 (\frac{k\pi}{l})^4 + m_3 (\frac{k\pi}{l})^2+ \mathcal{P} -[(\delta_1 + \delta_2)d + \delta_3(d -1)]\chi_2\eta(P^*)P^*(\frac{k\pi}{l})^2 \\
   &- \chi_2\eta(P^*)P^*d \left(2\dfrac{rS^*}{K} - \dfrac{mI^*P^*}{(a+I^*)^2}\right),
\end{align}
where
\begin{align*}
\mathcal{A} & = \dfrac{rS^*}{K}- \dfrac{mI^*P^*}{(a+I^*)^2},  \\
\mathcal{B} & = \dfrac{amP^*d}{(a+I^*)^2} - \dfrac{rmS^*I^*P^*}{K(a+I^*)^2}+ \left(\dfrac{r}{K} + \lambda \right)S^*I^* \lambda,  \\
\mathcal{C} & = \dfrac{rmaS^*P^*d}{K(a+I^*)^2}, \\
\mathcal{P} & = \mathcal{AB}-\mathcal{C} \\
 w_1 & = \delta_1 \delta_2 \delta_3 >0, \\
    w_2& =  \left[\delta_2 \delta_3\dfrac{rS^*}{K} - \delta_1 \delta_3 \dfrac{mI^*P^*}{(a+I^*)^2}\right], \\
    w_3 & =  \left[\dfrac{amP^*d}{(a+I^*)^2}\delta_1 +\left(- \dfrac{rmS^*I^*P^*}{K(a+I^*)^2} + \left(\dfrac{r}{K} + \lambda \right)\lambda S^*I^* \right)\delta_3 \right] , \\
     m_1 &=  \left[(\delta_1 + \delta_2 + \delta_3)(\delta_1 \delta_2 + \delta_2 \delta_3 + \delta_3 \delta_1) - \delta_1 \delta_2 \delta_3\right], \\
  m_2& =  -\{(a_{11}+a_{22})[\delta_3(\delta_1+\delta_2+\delta_3)+\delta_1\delta_2] + (a_{22}+a_{33})[\delta_2(\delta_1+\delta_2+\delta_3)+\delta_2\delta_3],\\
  & + (a_{11}+a_{33})[\delta_2(\delta_1+\delta_2+\delta_3)+\delta_1\delta_3]\},\\
 m_3& =  \delta_1[B-Aa_{22}+a_{23}a_{32}]+\delta_2[B-Aa_{11}]+\delta_3[B-A(a_{11}+a_{22})-(a_{11}a_{22}-a_{12}a_{21})].
\end{align*}
\begin{theorem}\label{thm:lasymcnd}
Suppose that \eqref{eq:eqmcnd},\eqref{eq:eqcnd2} hold. Then, the positive equilibrium solution $(S^*,I^*,P^*)$ of system \eqref{eq:model1} is locally asymptotically stable if $\chi_2$ satisfies the following inequalities.
 \begin{align}\label{eq:main1}
  \nonumber \dfrac{rmaS^*P^*d}{K(a+I^*)^2} > & \chi_2\eta(P^*)P^*\left(\delta_3 (\frac{k\pi}{l})^2 + \dfrac{rS^*}{K}d\right) ~~~~\text{and}\\ 
   \mathcal{P} > & \chi_2\eta(P^*)P^*\left(\dfrac{2rdS^*}{K}-\dfrac{mI^*P^*d}{(a+I^*)^2} + [(\delta_1 +\delta_2)d + \delta_3(d-1)](\frac{k\pi}{l})^2\right).
\end{align}
\end{theorem}
\begin{remark}
Based on the principle of linearized stability \cite[Theorem 5.2]{simonett1995center}, the positive equilibrium solution $E^*=(S^*,I^*,P^*)$ is asymptotically stable with respect to system \eqref{eq:model1} if and only if all eigenvalues of the matrix $-(\frac{k\pi}{l})^2D + J$ have negative real part, then according to the Routh-Hurwitz criterion \cite{liu2013pattern}, the positive equilibrium solution $E^*=(S^*,I^*,P^*)$ is locally asymptotically stable with respect to \eqref{eq:model1} iff the following conditions holds for each $k \in N^+$
\begin{equation*}
   \mathcal{A}_k >0, \mathcal{C}_k >0 ~\text{and}~ \mathcal{P}_k >0.
\end{equation*}
If one of the above conditions fail for some $k \in N^+$, then $E^*=(S^*,I^*,P^*)$ becomes unstable. Since  we always have $\mathcal{A}_k>0$, for all $k \in N^+$; moreover $\mathcal{C}_k>0$ and $\mathcal{P}_k >0$, follows from the condition \eqref{eq:main1}. Hence, $E^*=(S^*,I^*,P^*)$ is locally asymptotically stable for reaction-diffusion system \eqref{eq:model1} with prey-taxis provided \eqref{eq:main1} satisfied.
\end{remark}
From Theorem \ref{thm:lasymcnd}, we can know that system \eqref{eq:model1} has no Turing pattern for \eqref{eq:main1}. Thus, we can obtain the result on the instability of the positive equilibrium solution $E^*=(S^*,I^*,P^*)$ of system \eqref{eq:model1} if it fails to satisfy any condition of \eqref{eq:main1}. Furthermore, from Routh-Hurwitz criterion in \cite[corollary 2.2]{liu2013pattern}  for cubic polynomials, we find $\mathcal{C}_k=0$ for fixed $\chi_2$  satisfies \eqref{eq:steady}, then system \eqref{eq:model1} occurs the steady state bifurcation at $E^*=(S^*,I^*,P^*)$; Moreover, $\mathcal{P}_k=0$ for fixed $\chi_2$  satisfies \eqref{Hopf}, then system \eqref{eq:model1} undergoes a Hopf bifurcation at $E^*=(S^*,I^*,P^*)$.
We now seek existence conditions of the steady-state bifurcation and Hopf bifurcation respectively. According to Chen et al. \cite{chen2021stationary}, denote 
\begin{align*}
    S=&\{(k,\chi_2)\in R^+:\mathcal{C}_k=0 \}, \\
    H=&\{(k,\chi_2)\in R^+:\mathcal{P}_k =0\},
\end{align*}
be the steady-state bifurcation and Hopf bifurcation curves, respectively. Next, by considering the definitions of the curves 
S and H, with the prey-taxis coefficient 
$\chi_2$ as the bifurcation parameter, one proceeds to solve and express the solutions in the form

\begin{equation}\label{eq:steady}
       \chi^{S}_{2} = \frac{w_1 (\frac{k\pi}{l})^6 + w_2 (\frac{k\pi}{l})^4  + w_3 (\frac{k\pi}{l})^2 + \mathcal{C}}{\eta(P^*)P^*(rS^*d/K+\delta_3)}, 
    \end{equation}
    and \begin{equation}\label{Hopf}
      \chi^H_{2}= \dfrac{\mathcal{P}}{\eta(P^*)P^*\{[(\delta_1+\delta_2)d+\delta_3(d-1)](\frac{k\pi}{l})^2+ d(2rS^*/K-mI^*P^*/(a+I^*)^2)\} }.  \end{equation}
      
\subsection{Steady state bifurcation}
In this section, we employ abstract bifurcation theory to explore the existence of non-constant positive solutions for the following steady-state system.
\begin{equation}
 \begin{cases}
   \label{model2}
       \delta_1 \Delta S + rS\left( 1 - \dfrac{S+I}{K} \right) - \lambda SI = 0                                             & ~~~ x \in \Omega,              \\[10pt]
       \delta_2 \Delta I + \lambda SI - \dfrac{mPI}{a+ I} - \mu I  = 0                                                      & ~~~ x \in \Omega,         \\[10pt]
      \delta_3 \Delta P  - \chi_2\nabla \cdot(\eta(P) P \nabla I) + \dfrac{m I P}{a + I} - dP =0                   & ~~~ x \in \Omega,                \\[7pt]
     \dfrac{\partial S}{\partial \nu} =  \dfrac{\partial I}{\partial \nu} = \dfrac{\partial P}{\partial \nu} = 0         & ~~~ x \in \partial \Omega, \\[7pt]
     S(x,0)= S_0(x), I(x,0) = I_0(x), P(x,0) = P_0(x) & ~~~ x \in \Omega.
 \end{cases}
\end{equation}
In the remaining of this paper, unless explicitly stated otherwise, we assume that $\Omega = 
(0, l)$ with $l >0$. Under this assumption, system \eqref{model2} can be reformulated as follows:
\begin{equation}
 \begin{cases}
   \label{model3}
       \delta_1 S^{''} + rS\left( 1 - \dfrac{S+I}{K} \right) - \lambda SI = 0                                             & ~~~ x \in (0, l), \\[10pt]
       \delta_2  I^{''} + \lambda SI - \dfrac{mPI}{a+ I} - \mu I  = 0                                                      & ~~~ x \in (0, l),               \\[10pt]
      \delta_3 P^{''}  - \chi_2(\eta(P) P I^{'})^{'} + \dfrac{m I P}{a + I} - dP =0                   & ~~~ x \in (0, l),                   \\[7pt]
     S'(x)=I'(x)=P'(x)= 0         & ~~~ x=0,l .    \\[7pt]
 \end{cases}
\end{equation}

In order to utilize the bifurcation theorem proposed by Shi and Wang \cite{shi2009global}, we begin by defining the appropriate Hilbert space.
\begin{equation*}
    \mathcal{X} = \{\omega \in H^2(0,l)| \omega'(0)=\omega'(l)=0 \}.
\end{equation*}
Next, considering $\chi_2$ as the bifurcation parameter, we reformulate equation \eqref{model3} into an abstract framework.
\begin{equation*}
  \mathcal{F}(\chi_2,S,I,P)=0 , (\chi_2,S,I,P) \in \mathbb{R}\times \mathcal{X} \times \mathcal{X}\times \mathcal{X},  
\end{equation*}
where 
\begin{equation*}
    \mathcal{F}(\chi_2, S,I,P)= \left(\begin{array}{c}
         \delta_1 S^{''} + rS\left( 1 - \dfrac{S+I}{K} \right) - \lambda SI  \\
         \\
         \delta_2  I^{''} + \lambda SI - \dfrac{mPI}{a+ I} - \mu I \\
         \\
         \delta_3 P^{''}  - \chi_2(\eta(P) P  I^{'})^{'} + \dfrac{m I P}{a + I} - dP
    \end{array}\right).
\end{equation*}
It is easy to see that $\mathcal{F}(\chi_2, S^*,I^*,P^*)=0$ for any $\chi_2 \in \mathbb{R}$ and $\mathcal{F}: \mathbb{R}\times \mathcal{X}\times \mathcal{X} \times \mathcal{X} \rightarrow \mathcal{Y} \times \mathcal{Y}\times \mathcal{Y} $ is continuous and differentiable for $\mathcal{Y} = L^2(0,l)$. Furthermore, through direct computation, we find that for any fixed $(\tilde{S},\tilde{I},\tilde{P})\in \mathcal{X}\times \mathcal{X}\times \mathcal{X}$, the Fréchet derivative of $\mathcal{F}$ is given by 
\begin{align}\label{FrchetD}
    \nonumber & \mathcal{D}_{(S,I,P)}\mathcal{F}(\chi_2, \tilde{S},\tilde{I},\tilde{P})(S,I,P) \\[5pt]
    &= \left(\begin{array}{c}
         \delta_1 S^{''} + \mathcal{G}_1(S,I,P)  \\
         \\
         \delta_1 I^{''} + \mathcal{G}_2(S,I,P) \\
         \\
         \delta_1 P^{''} - \chi_2(\eta(\tilde{P})P \tilde{I}' + \eta(\tilde{P})\tilde{P}I' + \eta'(\tilde{P})\tilde{P}\tilde{I}'P)' +  \mathcal{G}_3(S,I,P)
    \end{array}\right),
\end{align}
where 
\begin{align*}
    \mathcal{G}_1(S,I,P)= & -\frac{r\tilde{S}}{K}S  -\left(\frac{r}{K} + \lambda \right)\tilde{S}I \\
    \mathcal{G}_2(S,I,P) = & \lambda \tilde{I} S + \left( \lambda \tilde{S} - \frac{a m \tilde{P}}{(a +\tilde{I})^2} - \mu \right) I - \frac{m \tilde{I}}{a + \tilde{I}} P \\
\mathcal{G}_3(S,I,P) = & \frac{a m \tilde{P}}{(a + \tilde{I})^2} I + \left( \frac{m \tilde{I}}{a + \tilde{I}} - d \right) P
\end{align*}
Now, we begin to find the possible bifurcation point with a particular value $\chi_2$. To this end, we should show that under that value of $\chi_2$, the Implicit Function Theorem fails on $\mathcal{F}$. Thus, $\mathcal{N}(\mathcal{D}(S,I,P)\mathcal{F}(\chi_2,S^*,I^*,P^*)) \neq \{0\}$ and this implies that there exists a nontrivial solution $(S,I,P)$ to the following problem:
\begin{equation}
 \begin{cases}
   \label{model4}
       \delta_1 S^{''} -\frac{rS^*}{K}S - \left(\frac{r}{K} + \lambda \right)S^*I = 0                                             & ~~~ x \in (0, l), \\[10pt]
       \delta_2  I^{''} +I^* \lambda S + \left( \lambda S^* - \frac{a m P^*}{(a + I^*)^2} - \mu \right) I - \frac{m I^*}{a + I^*} P  = 0                                                      & ~~~ x \in (0, l),               \\[10pt]
     -\chi_2 \eta(P^*)I^{''} + \delta_3 P^{''}+ \frac{a m P^*}{(a + I^*)^2} I + \left( \frac{m I^*}{a + I^*} - d \right) P =0                   & ~~~ x \in (0, l),                   \\[7pt]
     S'(x)=I'(x)=P'(x)= 0         & ~~~ x=0,l     \\[7pt]
 \end{cases}
\end{equation}
For any pair of functions $(S,I,P) \in \mathcal{X}\times \mathcal{X}\times \mathcal{X}$, $S, I$ and $P$ can be expanded as their eigen expansions:
\begin{equation*}
    \left(\begin{array}{c}
         S(x) \\
         I(x) \\
         P(x)
    \end{array}\right) = \sum\limits_{k=0}^{\infty}\left( \begin{array}{c}
         a_k  \\
         b_k   \\
         c_k
     \end{array}\right) cos \frac{k \pi x}{l},
\end{equation*}
where $cos \frac{k \pi x}{l}(k \in \mathbb{N})$ is all eigenfunctions corresponding to the eigenvalue $(\frac{k}{l})^2$ of $- \Delta $ under the Neumann boundary condition. Furthermore, at least one of these coefficients in the expansions is nonzero, if $(S,I,P)$ is nonzero. Since 
$k=0$ can be easily excluded, for $k \in \mathbb{Z}^+$, we substitute the eigenexpansions of $(S,I,P)$ into system \eqref{model4}. By multiplying system \eqref{model4} by $cos\frac{k\pi x}{l}$, integrating over $(0,l)$, and applying the boundary conditions, we find that
\begin{equation}
 \left( \begin{array}{ccc}
       -\delta_1 (\frac{k \pi}{l})^2 -\frac{rS^*}{K} & - \left(\frac{r}{K} + \lambda \right)S^* & 0 \\[10pt]
       I^* \lambda & -\delta_2 (\frac{k \pi}{l})^2 + \lambda S^* - \frac{amP^*}{(a+I^*)^2}-\mu & - \frac{mI^*}{a+I^*}             \\[10pt]
    0 &  \chi_2 \eta(P^*)(\frac{k \pi}{l})^2+\frac{amP^* }{(a+I^*)^2} & - \delta_3 (\frac{k \pi}{l})^2 + \left( \frac{m I^*}{a + I^*} - d \right)    
 \end{array}\right) \left(\begin{array}{c}
      a_k  \\
      b_k \\
      c_k
 \end{array}\right) = \left(\begin{array}{c}
      0  \\
      0   \\
      0
 \end{array}\right).
\end{equation} 
System \eqref{model4} has nontrivial solutions if and only if 
\begin{equation} \label{determ}
\begin{vmatrix}
    -\delta_1 (\frac{k \pi}{l})^2 -\frac{rS^*}{K} & - \left(\frac{r}{K} + \lambda \right)S^*   & 0 \\[10pt]
       I^* \lambda &-\delta_2 (\frac{k \pi}{l})^2 + \lambda S^* -\frac{amP^*}{(a+I^*)^2} - \mu & -\frac{mI^*}{a+I^*}             \\[10pt]
    0 &  \chi_2 \eta(P^*)(\frac{k \pi}{l})^2+\frac{amP^* }{(a+I^*)^2} & - \delta_3 (\frac{k \pi}{l})^2 + \left( \frac{m I^*}{a + I^*} - d \right)
\end{vmatrix}= 0,
\end{equation}
which implies that
\begin{align} \label{chi_k}
   \nonumber \chi_0= & \chi_2(k) \\
     \stackrel{\triangle}{=} & \frac{w_1 (\frac{k\pi}{l})^6 + w_2 (\frac{k\pi}{l})^4  + w_3 (\frac{k\pi}{l})^2 + \mathcal{C}}{\eta(P^*)P^*(rS^*d/K+\delta_3)}.
\end{align}
In the following, we will prove that all the conditions of Lemma 2 in \cite{song2017stability} are satisfied when $\chi_2= \chi_0$. Obviously, the condition $(i)$ of \cite[Lemma 2]{song2017stability} is satisfied.
Denoting $\mathbf{u}=(S,I,P)$, we can write \eqref{FrchetD} as
\begin{equation}\label{FrechetD2}
\mathcal{D}_{(S,I,P)}\mathcal{F}(\chi_2,\tilde{S},\tilde{I},\tilde{P})(S,I,P)= \mathbf{A}_0(\chi_2,\bold{u})\bold{u}^{''} + \mathbf{A}_1(\chi_2,\bold{u})\bold{u}^{'} + \bold{F}_0(\chi_2,\bold{u}),    
\end{equation}
where
\begin{align*}
  \bold{A}_0(\chi_2,\bold{u})= &\left(\begin{array}{ccc}
      \delta_1 & 0 & 0  \\
      \\
      0  & \delta_2 & 0 \\
      \\
      0 & -\chi_2(\eta(\tilde{P})\tilde{P} & \delta_3
  \end{array} \right),  \\[5pt]
  \bold{A}_1(\chi_2,\bold{u})= &\left(\begin{array}{ccc}
      0 ~~& 0 & 0  \\
      \\
      0 ~~ & 0 & 0 \\
      \\
      0 ~~& 0 & -\chi_2(\eta(\tilde{P})\tilde{I} + \eta'(\tilde{P})\tilde{P} \tilde{I}')
  \end{array} \right),  \\[5pt]
\bold{F}_0(\chi_2, \bold{u})= &\left(\begin{array}{c}
          \mathcal{G}_1(S,I,P) \\
          \\
         \mathcal{G}_2(S,I,P) \\
         \\
        -\chi_2((\eta(\tilde{P})\tilde{I}')'+ (\eta'(\tilde{P})\tilde{P}\tilde{I})' )P + \mathcal{G}_3(S,I,P)
    \end{array}\right).\\ 
\end{align*}

\noindent It is evident that \(\operatorname{Tr}(\mathbf{A}_0(\chi_2, \mathbf{u})) > 0\) and \(\det(\mathbf{A}_0(\chi_2, \mathbf{u})) > 0\) for all \(\mathbf{u} \in \mathcal{X} \times \mathcal{X} \times \mathcal{X}\). Consequently, the operator in \eqref{FrechetD2} is elliptic. In fact, it is strongly elliptic and satisfies Agmon's condition for angles \(\theta \in \left[-\frac{\pi}{2}, \frac{\pi}{2}\right]\), in line with Case 2 for \(N = 1\) discussed in \cite[Remark 2.5.5]{shi2009global}. Therefore, by \cite[Theorem 3.3 and Remark 3.4]{shi2009global}, the operator \(\mathcal{D}_{(S,I,P)}\mathcal{F}(\chi_2,\tilde{S},\tilde{I},\tilde{P})(S,I,P)\) is a Fredholm operator of index zero. As a result, condition (v) of \cite[Lemma 2]{song2017stability} is fulfilled.

 We now turn to verifying condition (iii) of \cite[Lemma 2]{song2017stability}. Specifically, we show that the null space \(\mathcal{N}(\mathcal{D}_{(S,I,P)}\mathcal{F}(\chi_0, S^*, I^*, P^*))\) is one-dimensional and is spanned by the following function

 \begin{equation} \label{null_space}
    \mathcal{N}(\mathcal{D}_{(S,I,P)}\mathcal{F}(\chi_0,S^*,I^*,P^*) = \text{span}\{\Tilde{S}_k,\Tilde{I}_k,\Tilde{P}_k\}, 
 \end{equation}
where 
\begin{equation*}
    (\Tilde{S}_k,\Tilde{I}_k,\Tilde{P}_k) = \{\Lambda_k, \zeta_k,1\}cos \frac{k\pi x}{l}
\end{equation*}
where
\begin{align*}
    \Lambda_k= & \frac{\left(\delta_1 (\frac{k\pi}{l})^2 - \frac{mI^*P^*}{(a+I^*)^2}\right)\delta_3 (\frac{k\pi}{l})^2}{I^*\lambda\left(\chi_2 \eta (P^*)(\frac{k\pi}{l})^2+ \frac{amP^*}{(a+I^*)^2}\right)} + \frac{d}{I^*\lambda}\\[5pt]
    \zeta_k= &\frac{\delta_3 (\frac{k\pi}{l})^2}{\chi_2 \eta (P^*)(\frac{k\pi}{l})^2+ \frac{amP^*}{(a+I^*)^2}}
\end{align*}\\

Because we have proved that $\mathcal{D}_{(S,I,P)}\mathcal{F}(\chi_2,S,I,P)$ is a Fredholm operator of index 0, it follows that $\text{dim} \mathcal{R}(\mathcal{D}_{(S,I,P)}\mathcal{F}(\chi_0,S^*,I^*,P^*))=1$. Thus, the condition (iii) of \cite[Lemma 2]{song2017stability}  is satisfied.
Finally,we prove the condition (iv) of \cite[Lemma 2]{song2017stability}  is satisfied. Define
\begin{equation*}
    y_0= \left(\Lambda_k \text{cos} \frac{k\pi x}{l}, \zeta_k \text{cos} \frac{k\pi x}{l}, \text{cos} \frac{k\pi x}{l}\right).
\end{equation*}
From \eqref{FrchetD}, we have 
\begin{equation}\label{Dxeq}
  \frac{d}{d \chi_2}(\mathcal{D}_{(S,I,P)}\mathcal{F}(\chi_2, \tilde{S},\tilde{I},\tilde{P})(S,I,P))= \left(\begin{array}{c}
         0  \\
         0 \\
          - (\eta(\tilde{P})P \tilde{I}' + \eta(\tilde{P})\tilde{P}I' + \eta'(\tilde{P})\tilde{P}\tilde{I}'P)'
    \end{array}\right)   
\end{equation}
which is continuous. It means that the condition (ii) of \cite[Lemma 2]{song2017stability}  is satisfied. Thus, from \eqref{Dxeq}, we can get 
\begin{equation*}
 \frac{d}{d \chi_2} (\mathcal{D}_{(S,I,P)}\mathcal{F}(\chi_0, \tilde{S},\tilde{I},\tilde{P})(S,I,P))= \left(\begin{array}{c}
         0  \\
         0 \\
          - \eta(P^*)P^* I''
    \end{array}\right),   
\end{equation*}
which yields 
\begin{equation*}
 \frac{d}{d \chi_2}( \mathcal{D}_{(S,I,P)}\mathcal{F}(\chi_0, S^*,I^*,P^*)y_0)= \left(\begin{array}{c}
         0  \\
         0 \\
         \eta(P^*)P^* \zeta_k (\frac{k\pi}{l})^4
    \end{array}\right)cos \frac{k\pi x}{l}.   
\end{equation*}\\
Suppose that 
\begin{equation*}
   \frac{d}{d \chi_2} (\mathcal{D}_{(S,I,P)}\mathcal{F}(\chi_0, S^*,I^*,P^*)y_0) \in \mathcal{R}(\mathcal{D}_{(S,I,P)}\mathcal{F}(\chi_0, S^*,I^*,P^*)),
\end{equation*}
which implies that there must exist a nontrivial solution $(S,I,P)$ such that
\begin{equation*}
    \mathcal{D}_{(S,I,P)}\mathcal{F}(\chi_0, S^*,I^*,P^*)(S,I,P) = \frac{d}{d \chi_2} (\mathcal{D}_{(S,I,P)}\mathcal{F}(\chi_0, S^*,I^*,P^*)y_0),
\end{equation*}
that is
\begin{equation}
 \begin{cases}
   \label{model5}
       \delta_1 S^{''} -\frac{rS^*}{K}S - \left(\frac{r}{K} + \lambda \right)S^*I = 0                                             & ~~~ x \in (0, l), \\[10pt]
       \delta_2  I^{''} +I^* \lambda S + \left( \lambda S^* - \frac{a m P^*}{(a + I^*)^2} - \mu \right) I - \frac{m I^*}{a + I^*} P  = 0                                                      & ~~~ x \in (0, l),               \\[10pt]
     -\chi_2 \eta(P^*)I^{''} + \delta_3 P^{''}+\frac{a m P^*}{(a + I^*)^2} I + \left( \frac{m I^*}{a + I^*} - d \right) P  = \eta(P^*)P^* \zeta_k (\frac{k\pi}{l})^4cos \frac{k\pi x}{l}                   & ~~~ x \in (0, l),                   \\[7pt]
     S'(x)=I'(x)=P'(x)= 0         & ~~~ x=0,l     \\[7pt]
 \end{cases}
\end{equation}.
Multipying the three equations in system \eqref{model5} by cos $\frac{k\pi x}{l}$ and then integrating them over $(0,l)$ by parts, we obtain
\begin{equation} \label{multip}
\begin{aligned}
    &\left(\begin{array}{ccc}
        -\delta_1 \left(\frac{k\pi}{l}\right)^2-\frac{rS^*}{K} & \left(\frac{r}{K} + \lambda \right)S^* & 0\\
        \\
         I^*\lambda & -\delta_2 \left(\frac{k\pi}{l}\right)^2+ \lambda S^* -\frac{amP^*}{(a+I^*)^2} - \mu & -\frac{mI^*}{a+I^*} \\
         \\
         0 &  \chi_2 \eta(P^*)\left(\frac{k \pi}{l}\right)^2+\frac{amP^* }{(a+I^*)^2} & - \delta_3 \left(\frac{k \pi}{l}\right)^2+\left( \frac{m I^*}{a + I^*} - d \right) 
    \end{array}\right)
    \left(\begin{array}{c}
         \int_{0}^{l} S\cos \frac{k\pi x}{l} \, dx  \\
         \\
         \int_{0}^{l} I\cos \frac{k\pi x}{l} \, dx\\
         \\
         \int_{0}^{l} P \cos \frac{k\pi x}{l} \, dx 
    \end{array}\right) \\
    &= \left(\begin{array}{c}
         0  \\
         \\
         0\\
         \\
         \dfrac{\eta(P^*)P^*\zeta_k k^4 \pi^4}{2l^3}
    \end{array}\right).
\end{aligned}
\end{equation}

However, \eqref{multip} is impossible because the determinant of the coefficient matrix on the left-hand side is zero by \eqref{determ}. Thus, we obtain the contradiction and it follows that 
\begin{equation*}
   \frac{d}{d \chi_2} (\mathcal{D}_{(S,I,P)}\mathcal{F}(\chi_0, S^*,I^*,P^*)y_0) \notin \mathcal{R}(\mathcal{D}_{(S,I,P)}\mathcal{F}(\chi_0, S^*,I^*,P^*)),
\end{equation*}\\
i.e., the condition (iv) of \cite[Lemma 2]{song2017stability} is satisfied. Thus, we have shown that all the conditions of \cite[Lemma 2]{song2017stability} are satisfied and obtain the following result.
\begin{theorem} \label{Theorem}
    Assume that $(\mathbf{H_1})$, $(\mathbf{H_2})$ holds. Suppose that for positive integers $k,j \in \mathbb{Z^+}$, $$\chi^S_{2k} \neq \chi^S_{2j}, \forall ~ k \neq j$$ where $\chi^S_k$ is given in the equation \eqref{chi_k}, $E^*=(S^*,I^*,P^*)$ is the positive equilibrium of system \eqref{eq:model1}. Then for each $k \in \mathbb{Z^+}$, there exists a small constant $\rho >0$ and a continuous function $
     \varepsilon \in (-\rho , \rho):\rightarrow (\chi_2(\varepsilon),S_k(\varepsilon,x),I_k(\varepsilon,x),P_k(\varepsilon,x))\in \mathbb{R^-}\times \mathcal{X}\times \mathcal{X} \times \mathcal{X} ~~\text{such that}$
     \begin{align}\label{ConTh1}
         &\chi_2(\varepsilon) = \chi_0 + \bold{O}(\varepsilon), \varepsilon \in (-\rho , \rho),\\  
&\nonumber (S_k(\varepsilon,x),I_k(\varepsilon,x),P_k(\varepsilon,x)) = (S^*,I^*,P^*) + \varepsilon (\Tilde{S}_k,\Tilde{I}_k,\Tilde{P}_k)+ \bold{O}(\varepsilon^2), \varepsilon \in (-\rho , \rho) ,
     \end{align}
     where $(\Tilde{S}_k,\Tilde{I}_k,\Tilde{P}_k)$ is given in equation \eqref{null_space} and $\bold{O}(\varepsilon^2) \in \mathcal{Z}$ is in the closed complement of 
      $\mathcal{N}(\mathcal{D}_{(S,I,P)}\mathcal{F}(\chi_0,S^*,I^*,P^*)$ is defined by
     \begin{equation}\label{ConTh2}
       \mathcal{Z} = \left\{(S,I,P) \in \mathcal{X}\times \mathcal{X} \times \mathcal{X}| \int_{0}^{l} (S\widehat{S}_k+I\widehat{I}_k+P\widehat{P}_k) \,dx = 0 \right\}.  
     \end{equation}
     \end{theorem}

\subsection{Stability analysis of boundary equilibrium - Finite time extinction}
\subsubsection{The $p=2,1>\gamma>0$ case}

Since existence of weak solution to \eqref{eq:model1} has been proven via Theorem \ref{thm:mt1}, we now consider the long time dynamics of the solution, certain components of which will reach states in short (finite) time. We first consider the case where diffusion is regular ($p=2$), but the mortality of the infected prey is heightened ($ 0 < \gamma <1$). We state the following theorem to this end.
\begin{theorem}\label{thm:FFTE}
Consider the model \eqref{eq:pde_modeld2}, with $p=2$.
There exists positive initial data $(S_0(x),I_0(x),P_{0}(x))$ and $0<\gamma<1$ such that for 
any set of positive parameters 

$(\delta_{1},\delta_{2},\delta_{3},\lambda, K, r, m,a, \chi_{1},\chi_{2},d)$ the solution $(S,I,P) \to (K,0,0)$ in $L^{2}(\Omega)$, with $I \rightarrow 0$ in finite time.
\end{theorem}

\begin{proof}
 We now multiply the $I$ equation in \eqref{eq:model1} by $I$, and integrate by parts over $\Omega$ to obtain

\begin{eqnarray}
\frac{1}{2} \dfrac{d}{dt} ||I||^{2}_{2} + \delta_{1} || I_{x} ||^{2}_{2} + \mu\int_{\Omega} I^{1+\gamma} dx +  \int_{\Omega} \frac{mPI^{2}}{a+I} dx &=&  \lambda\int_{\Omega}SI^{2} dx.    \nonumber \\
\end{eqnarray}

 Thus, we have via positivity that

\begin{equation}
\frac{1}{2} \dfrac{d}{dt} ||I||^{2}_{2} + C_{4}  \left( || I_{x} ||^{2}_{2} + \int_{\Omega} I^{1+\gamma}dx \right)  \leq \lambda K||I||^{2}_{2} dx,    \nonumber \\
\end{equation}
where $C_{4} = \min{(\delta_{1}, \mu)}$. We aim to show,  

\begin{equation}
 \left( || I||^{2}_{2}  \right) ^{\alpha}  \leq  C_{4}  \left( || I_{x} ||^{2}_{2} + \int_{\Omega} I^{1+\gamma}dx \right),
\end{equation}
where $0< \alpha < 1$. This will yield finite time extinction of the solution $I$, similar to the ODE

\begin{equation}
\frac{dy}{dt} = \lambda K y- C_{4}y^{\alpha}, 0< \alpha < 1, \lambda, K,  C_{4} > 0.
\end{equation}
Now we apply the Gagliardo-Nirenberg-Sobolev (GNS) inequality via lemma \ref{lem:gns}, \cite{friedman2008partial}, by considering exponents such that,

\begin{equation}
W^{k,p^{'}}(\Omega) = L^{2}(\Omega), \ W^{m,q^{'}}(\Omega) = W^{1,2}(\Omega), \ L^{q}(\Omega) = L^{1+\gamma}(\Omega),
\end{equation}
for $0 < \gamma < 1$.
This yields

\begin{equation}
\label{eq:uest}
	||I|_{L^{2}(\Omega)} \leq C ||I||^{\theta}_{W^{1,2}(\Omega)} ||I||^{1 - \theta}_{L^{q}(\Omega)}, 
	\end{equation}
 as long as

%

	\begin{equation}
	\label{eq:1h}
	\frac{2-q}{2+q} \leq \theta \leq 1.
	\end{equation}
We raise both sides of \eqref{eq:uest} to the power of $l$, $0<l<2$, to obtain

%
	
	\begin{equation}
	\left( \int_{\Omega} I^{2}dx \right)^{\frac{l}{2}} \leq C\left( \int_{\Omega}  (I_{x}) ^{2}dx \right)^{\frac{l \theta}{2}} \left( \int_{\Omega}  I^{q}dx \right)^{\frac{l (1-\theta)}{q}}.
	\end{equation}
	Using Young's inequality on the right hand side (for $a b \leq \frac{a^{r}}{r} + \frac{b^{m}}{m}$), with $r = \frac{2}{l \theta}, \ m= \frac{q}{l (1-\theta)}$, yields
	
	\begin{equation}
	\left( \int_{\Omega} I^{2}dx \right)^{\frac{l}{2}} \leq C \left( \int_{\Omega}  (I_{x}) ^{2}dx + \int_{\Omega}  I^{q}dx \right).
	\end{equation}
	We notice that given any $1<q=1+\gamma <2$, it is always possible to choose $0<l<2$, such that, $\frac{1}{r} + \frac{1}{m} = 1$,
	
	\begin{equation}
\frac{1}{r} + \frac{1}{m} = 	\frac{l \theta}{2}  + \frac{l (1-\theta)}{q} = 1,
	\end{equation}
	by choosing
	
	\begin{equation}
	\theta  = \frac{\frac{1}{l} - \frac{1}{q} }{\frac{1}{q} - \frac{1}{2} } = \frac{2(q-l)}{l(2-q)},
	\end{equation} 
	thus we need to choose $l$ such that,
	
	\begin{equation}
	 \frac{2(q-l)}{l(2-q)} \geq \frac{2-q}{2+q}
	\end{equation} 
	or
	
	\begin{equation}
	 \frac{1}{l} \geq \frac{(2-q)^{2}}{2q(2+q)} + \frac{1}{2q}.
	\end{equation} 
This is in the case $n=1$. For $n=2$ case, given $1<q^{*}<2$, we need to choose $l$ such that,

\begin{equation}
	 \frac{2(q^{*}-l)}{l(2-q^{*})} \geq \frac{2-q^{*}}{2},
	\end{equation} 

    or

    \begin{equation}
    \label{eq:el2}
	 \frac{1}{l} \geq \frac{(2-q^{*})^{2}}{4q^{*}} + \frac{1}{q^{*}},
	\end{equation} 
    
	whilst again being mindful that $0 \leq \theta \leq 1$.
This enables the application of Young's inequality above, within the required restriction \eqref{eq:1h}, enforced by the GNS inequality when $n=1,2$. Now, if it turns out that $C < C_{4}$, then

\begin{equation}
	\left( \int_{\Omega} I^{2}dx \right)^{\frac{l}{2}} \leq C \left( \int_{\Omega}  (I_{x}) ^{2}dx + \int_{\Omega}  I^{q}dx \right) < C_{4} \left( \int_{\Omega}  (I_{x}) ^{2}dx + \int_{\Omega}  I^{q}dx \right).
	\end{equation}

Thus, we have
\begin{equation}
\frac{1}{2} \dfrac{d}{dt} ||I||^{2}_{2} + C_{4} \left(  ||I||^{2}_{2} \right)^{\frac{l}{2}}   \leq \lambda K ||I||^{2}_{2} dx. \nonumber \\
\end{equation}

If not, the $C_{4}$ in the equation above is adjusted accordingly.
Let  $\alpha=\frac{l}{2} <1$. We have that $||I||^{2}_{2} \rightarrow 0$ as $t \rightarrow T^{*} < \infty$, for sufficiently chosen initial data, similar to the ODE,

\begin{equation}
\frac{dy}{dt} = \lambda K y - C_{4}y^{\alpha}, 0< \alpha < \lambda,  K, C_{4} > 0.
\end{equation}
This completes the proof.
\end{proof}

\begin{corollary}
\label{cor:sdc}
Consider the model \eqref{eq:pde_modeld2}, with $p=2, \gamma =1, n=2$. There exists parameters $(\delta_{1},\delta_{2},\delta_{3},\lambda, K, r, m,a, \chi_{1},\chi_{2},d)$, under the assumptions of Theorem \ref{thm:lasymcnd}, for which the interior equilibrium $(S^{*},I^{*},P^{*})$ is locally stable. However, for the same parameters with $0<\gamma<1$, if the initial data
 satisfies 
 \begin{equation}
 ||I_{0}||^{2}_{2} < \left(\frac{C_{4}}{\lambda K} \right)^{\left(\frac{1}{1-\left(  \frac{2(1+\gamma)}{(1-\gamma)^{2}+4}\right)}\right)},
 \end{equation}
  the solution $(S,I,P) \rightarrow (K,0,0)$ in $L^{2}(\Omega)$. 
\end{corollary}

\begin{remark}
This tells us that there is initial data $(S_0(x),I_0(x),P_{0}(x))$, that will be attracted to $(S,I,P)$, when with $p=2, \gamma =1$, but with sufficiently chosen $0<\gamma<1$, the same initial condition could be attracted to $(K,0,0)$.
\end{remark}

\begin{proof}
Consider the estimate for $I$,

\begin{equation}
\frac{1}{2} \dfrac{d}{dt} ||I||^{2}_{2} +  C_{4}\left( || I||^{2}_{2}  \right) ^{\alpha} \leq  \lambda K||I||^{2}_{2}. 
\end{equation}
Choosing, $y = ||I||^{2}_{2} $ yields
\begin{equation}
\frac{dy}{dt} \leq \lambda K y - C_{4} y^{\alpha},
\end{equation}

We will consider the ODE 

\begin{equation}
\frac{dV}{dt} = \lambda K V - C_{4} V^{\alpha}, \ V(0)=v_{0}.
\end{equation}

Note via simple comparison, $y \leq V$, for all time $t$.
Set $V=\frac{1}{U}$, then

\begin{equation}
\frac{dV}{dt} =  -\frac{1}{U^{2}}\frac{dU}{dt} =   \lambda K \frac{1}{U}  - \frac{C_{2}}{U^{\alpha}}, \ U(0)=\frac{1}{V(0)}.
\end{equation}

This yields

\begin{equation}
\frac{dU}{dt}  =  -\lambda K U + C_{4} U^{2-\alpha}, \ U(0)=\frac{1}{V(0)}, \ 0<\alpha<1
\end{equation}

showing that $U$ will blow up in finite time for sufficiently large initial condition $U(0)$. That is if $U(0)$ is chosen such that

\begin{equation}
C_{4}(U(0))^{2-\alpha} > \lambda K U(0), 
\end{equation}

then 

\begin{equation}
\lim_{t \rightarrow T^{*} < \infty} U(t) = \infty.
\end{equation}

However, since $V=\frac{1}{U}$, we have
\begin{equation}
\lim_{t \rightarrow T^{*} < \infty} V(t) = \lim_{t \rightarrow T^{*} < \infty} \frac{1}{U(t)} = \frac{1}{\lim_{t \rightarrow T^{*} < \infty} U(t)} \rightarrow 0.
\end{equation}

That is, $V$ will go extinct in finite time, for initial data small enough, that is given by,

\begin{equation}
(V(0))^{1-\alpha}< \frac{C_{4}}{\lambda K}.
\end{equation}

Since $y\leq V$, $I$ will go extinct in finite time as well. 
That is,

\begin{equation}
\lim_{t \rightarrow T^{**} < \infty} y(t) \rightarrow 0
\end{equation}

for initial condition chosen such that,

\begin{equation}
(y(0))^{1-\alpha}< \frac{C_{4}}{\lambda K}.
\end{equation}

We derive the relation between $\alpha$ and $\gamma$. We will consider the case $n=2$. From \eqref{eq:el2}, it follows that

\begin{equation}
    \frac{1}{\alpha} = \frac{(1-\gamma)^{2}}{2(1+\gamma)} + \frac{2}{1+\gamma} => \alpha = \frac{2(1+\gamma)}{(1-\gamma)^{2}+4}.
\end{equation}

This implies if 

\begin{equation}
||I_{0}(x)||^{2}_{2} \leq \left(\frac{C_{4}}{\lambda K}\right)^{\left(\frac{1}{1-\left(\frac{2(1+\gamma)}{(1-\gamma)^{2}+4}\right)}\right)}, 
\end{equation}

then $I$ will go extinct in finite time, from which the convergence of $S$ to $K$ and $P$ to $0$ in infinite time follows. This proves the lemma.
\end{proof}

\begin{remark}
We now consider the case $p>2,1>\gamma>0$ - we will give a result on how the initial condition in this case compares to helping/inhibiting FTE in the case $p=2,1>\gamma>0$.
    \end{remark}

\subsubsection{The $p>2,0<\gamma<1$ case}

We now investigate the effect of ``slow" dispersal on the
finite time extinction dynamics.
\begin{theorem}\label{thm:FFTEsd}
Consider the model system \eqref{eq:pde_modeld2}, with $p>2$.
There exists positive initial data $(S_0(x),I_0(x),P_{0}(x))$ and $0<\gamma<1$ such that for 
any set of positive parameters 

$(\delta_{1},\delta_{2},\delta_{3},\lambda, K, r, m,a, \chi_{1},\chi_{2},d)$ the solution $(S,I,P) \to (K,0,0)$ in $L^{2}(\Omega)$, with $I \rightarrow 0$ in finite time.

\end{theorem}

\begin{proof}
 We proceed similarly and multiply the $I$ equation in \eqref{eq:model1} by $I$ and integrate by parts to obtain

\begin{eqnarray}
\frac{1}{2} \dfrac{d}{dt} ||I||^{2}_{2} + \delta_{1} || I_{x} ||^{p}_{p} + \mu\int_{\Omega} I^{1+\gamma} dx +  \int_{\Omega} \frac{mPI^{2}}{a+I} dx &=&  \lambda\int_{\Omega}SI^{2} dx,    \nonumber \\
\end{eqnarray}

yielding

\begin{equation}
\frac{1}{2} \dfrac{d}{dt} ||I||^{2}_{2} + C_{4}  \left( || I_{x} ||^{p}_{p} + \int_{\Omega} I^{1+\gamma}dx \right)  \leq \lambda K||I||^{2}_{2} dx.    \nonumber \\
\end{equation}
We again aim to show that 

\begin{equation}
 \left( || I||^{2}_{2}  \right) ^{\alpha}  \leq  C_{4}  \left( || I_{x} ||^{p}_{p} + \int_{\Omega} I^{1+\gamma}dx \right),
\end{equation}
where $0< \alpha < 1$, If so then we will have finite time extinction of the solution $I$. As earlier choosing exponents such that

\begin{equation}
W^{k,p^{'}}(\Omega) = L^{2}(\Omega), \ W^{m,q^{'}}(\Omega) = W^{1,p}(\Omega), \ L^{q}(\Omega) = L^{1+\gamma}(\Omega),
\end{equation}
for $p>2, \ 0 < \gamma < 1$. This yields

\begin{equation}
\label{eq:uest}
	||I|_{L^{2}(\Omega)} \leq C ||I||^{\theta}_{W^{1,p}(\Omega)} ||I||^{1 - \theta}_{L^{q}(\Omega)}, 
	\end{equation}
 as long as in the $n=2$ case,

	\begin{equation}
	\label{eq:1h}
	\frac{\frac{1}{1+\gamma} - \frac{1}{2}}{\frac{1}{1+\gamma} + \frac{1}{2} - \frac{1}{p}} = \theta \leq 1.
	\end{equation}
We raise both sides of \eqref{eq:uest} to the power of $l$, $0<l<2$, to obtain

%
	
	\begin{equation}
	\left( \int_{\Omega} I^{2}dx \right)^{\frac{l}{2}} \leq C\left( \int_{\Omega}  |\nabla I| ^{p}dx \right)^{\frac{l \theta}{p}} \left( \int_{\Omega}  I^{q}dx \right)^{\frac{l (1-\theta)}{q}}.
	\end{equation}
	Using Young's inequality on the right hand side (for $a b \leq \frac{a^{r}}{r} + \frac{b^{m}}{m}$), with $r = \frac{p}{l \theta}, \ m= \frac{q}{l (1-\theta)}$, yields
	
	\begin{equation}
	\left( \int_{\Omega} |I|^{2}dx \right)^{\frac{l}{2}} \leq C \left( \int_{\Omega}  |\nabla I|^{p}dx + \int_{\Omega}  |I|^{q}dx \right).
	\end{equation}
	We notice that given any $1<q=1+\gamma <2$, it is always possible to choose $0<l<2$, such that, $\frac{1}{r} + \frac{1}{m} = 1$,for $n=2$ case, being mindful that $0 \leq \theta \leq 1$. To demonstrate, we choose $\theta = \frac{\frac{1}{1+\gamma} - \frac{1}{2}}{\frac{1}{1+\gamma} + \frac{1}{2} - \frac{1}{p}}$, and see it follows that

    \begin{equation}
\frac{1}{2} - \frac{1}{1+\gamma} < \frac{\left(\frac{1}{1+\gamma} - \frac{1}{2} \right) \left( \frac{1}{p} -  \frac{1}{1+\gamma}\right)}{\left(\frac{1}{1+\gamma} - \frac{1}{p} + \frac{1}{2} \right)},
    \end{equation}
    
    demonstrating a choice of $0<l = \frac{1}{\frac{\theta}{p} + \frac{1-\theta}{1+\gamma}}<2$ is possible.
    This enables the application of Young's inequality above, within the required restriction \eqref{eq:1h}, enforced by the GNS inequality when $n=2$. Thus, we have
    
\begin{equation}
\frac{1}{2} \dfrac{d}{dt} ||I||^{2}_{2} + C_{4} \left(  ||I||^{2}_{2} \right)^{\frac{l}{2}}   \leq \lambda K ||I||^{2}_{2} dx. \nonumber \\
\end{equation}

Let  $\alpha=\frac{l}{2} <1$, and FTE occurs as earlier in Theorem \ref{thm:FFTE}.  
This completes the proof.
\end{proof}

We now compare the case $p>2,1>\gamma>0$, to the $p=2, \ 1>\gamma>0$ case - we comment on how the initial condition in this case compares to helping/inhibiting FTE, via the following lemma,

\begin{lemma}
\label{lem:l1da}
Consider system \eqref{eq:model1}, and two thresholds for initial data, $(i) \frac{C_{4}}{\lambda K} > 1$, $(ii) \frac{C_{4}}{\lambda K} < 1$. In (i) there could be initial data $I_{0}$ that is sufficient to cause FTE to the $(K,0,0)$ state, when $p>2$ or there is slow diffusion, but not when $p=2$. In (ii) there could be initial data $I_{0}$ that is sufficient to cause FTE to the $(K,0,0)$ state when $p=2$ or when there is regular diffusion, but not when $p>2$.
\end{lemma}

\begin{proof}
When $p>2$, we have from Theorem \ref{thm:FFTEsd} and \eqref{eq:1h} in particular, that $l=\frac{1}{\frac{\theta}{p} + \frac{1-\theta}{1+\gamma}} > \frac{1}{\frac{\theta}{2} + \frac{1-\theta}{1+\gamma}} $, which is when $p=2$. Thus, since $||I_{0}||^{2}_{2} < \left( \frac{C_{4}}{\lambda K} \right)^{\frac{1}{1-\frac{l}{2}}}$ causes FTE extinction. Clearly,

\begin{equation}
    \frac{1}{1-\left(\frac{1}{\frac{\theta}{p} + \frac{1-\theta}{1+\gamma}} \right)} > \frac{1}{1-\left(\frac{1}{\frac{\theta}{2} + \frac{1-\theta}{1+\gamma}} \right)} ,
\end{equation}
so, if $\frac{C_{4}}{\lambda K} > 1$, for FTE we require
\begin{equation}
\label{eq:2bd}
 ||I_{0}||^{2}_{2} < \left( \frac{C_{4}}{\lambda K} \right)^{\frac{1}{1-\left(\frac{1}{\frac{\theta}{2} + \frac{1-\theta}{1+\gamma}} \right)} } < \left( \frac{C_{4}}{\lambda K} \right)^{\frac{1}{1-\left(\frac{1}{\frac{\theta}{p} + \frac{1-\theta}{1+\gamma}} \right)}}.
\end{equation}

Therefore, in this case, there could be initial data $I_{0}$ that is sufficient to cause FTE to the $(K,0,0)$ state when $p>2$ or there is slow diffusion, but not when $p=2$. To see this, just consider initial data $I_{0}$ such that $||I_{0}||^{2}_{2}$ is exactly in between the two bounds in \eqref{eq:2bd}.

On the other hand if $\frac{C_{4}}{\lambda K} < 1$, then for FTE, we would have

\begin{equation}
\label{eq:2bd1}
 ||I_{0}||^{2}_{2} < \left( \frac{C_{4}}{\lambda K} \right)^{\frac{1}{1-\left(\frac{1}{\frac{\theta}{p} + \frac{1-\theta}{1+\gamma}} \right)}} < \left( \frac{C_{4}}{\lambda K} \right)^{\frac{1}{1-\left(\frac{1}{\frac{\theta}{2} + \frac{1-\theta}{1+\gamma}} \right)} }. 
\end{equation}

So in this case, there could be initial data $I_{0}$ that is sufficient to cause FTE to the $(K,0,0)$ state when $p=2$ or when there is regular diffusion, but not when $p>2$, the slow diffusion case. To see this again, just consider initial data $I_{0}$ such that $||I_{0}||^{2}_{2}$ is exactly in between the two bounds in \eqref{eq:2bd1}.
\end{proof}

\subsection{Thresholds for disease spread in temporal case with $0<\gamma <1$}

Here we consider the temporal model, so setting $\delta_{1}=\delta_{2}=\delta_{3}=\chi_{1}=\chi_{2}=0$, $\gamma =1$, in system \eqref{eq:model1}. 

\begin{definition}
\label{def:d22}
    A disease is persistent if 
    \begin{equation}
        \exists \epsilon > 0, \ I_{0} > 0, \ => \liminf_{t \rightarrow \infty}{I(t)} \geq \epsilon.
    \end{equation}
\end{definition}

It is well known from \cite{chattopadhyay2001pelicans} that the conditions for disease persistence in prey only is $\lambda K > \mu$ and persistence of disease in prey and persistence of all species is given by \eqref{eq:eqmcnd}, \eqref{eq:eqcnd2}. These essentially are,

\begin{equation*}
\lambda K > \mu , \  m> \max [d+\lambda ad(r+\lambda K)/{r(\lambda K - \mu)}, d\lambda K/r].
\end{equation*}


 We will simplify the above and consider the case when \begin{equation}
     \frac{d\lambda K}{r} > d+\frac{\lambda ad(r+\lambda K)}{r(\lambda K - \mu)}.
 \end{equation}
 This yields the following as a simplified condition for persistence of all populations,

\begin{equation}
\label{eq:pcond}
\boxed{m \left(\frac{r}{d}\right) > \lambda K > \mu, \ S(0) > \frac{\mu}{\lambda K} }.
\end{equation}

We show how this changes in the event that $0<\gamma<1$.
We state the following lemma to this end.
\begin{lemma}
\label{lem:tl1}
Consider the temporal version of \eqref{eq:model1}, when $\delta_{1}=\delta_{2}=\delta_{3}=\chi_{1}=\chi_{2}=0$, $\gamma =1$, and $\frac{d\lambda K}{r} > d+\frac{\lambda ad(r+\lambda K)}{r(\lambda K - \mu)}, \ \ \lambda K > \mu, \ S(0) > \frac{\mu}{\lambda K}$. Then if, 

\begin{equation}
R_{0}=\frac{m r}{d\lambda K} > 1. 
\end{equation}

The disease persists in the population, and the interior equilibrium $(S^{*},I^{*},P^{*})$ is locally asymptotically stable. However if $0<\gamma <1$, and  
\begin{equation}
|I_{0}| < \left(\frac{\mu}{\lambda K}\right)^{\frac{1}{1-\gamma}}.
\end{equation}
Then $I \rightarrow 0$ in finite time, and the disease cannot persist.
\end{lemma}

\begin{proof}
In the case, $\gamma=1$, by defining $R_{0}=\frac{m r}{d\lambda K}$, we see that utilizing results from \cite{chattopadhyay2001pelicans} we have $m \left(\frac{r}{d}\right) > \lambda K > \mu$, yields disease persistence. This implies $ R_{0}>1$, yields the persistence condition. Now if $0<\gamma <1$, from the equation for $I$, using positivity, and bound on $S$,

\begin{equation}
\frac{dI}{dt} \leq \lambda K I - \mu I^{\gamma}.
\end{equation}

Thus similar to Theorems \ref{thm:FFTE}-\ref{thm:FFTEsd}, we have that if 

\begin{equation}
|I_{0}| < \left(\frac{\mu}{\lambda K}\right)^{\frac{1}{1-\gamma}}.
\end{equation}
Then $I \rightarrow 0$ in finite time, and the disease cannot persist. This proves the lemma.
\end{proof}

\begin{remark}
This shows that the classical definition of disease persistence via definition \ref{def:d22}, will not apply in the case of increased mortality due to FTE, when $\gamma <1$. In this case there always exist positive (but possibly small) initial condition which will lead to disease extinction. Thus in this setting one may need to define a threshold $C$ for the initial data, to avoid extinction.
\end{remark}
To this end, we define
\begin{definition}
    A disease is persistent if 
    \begin{equation}
        \exists \epsilon > 0, \ C>0 \ I_{0} > C \ => \liminf_{t \rightarrow \infty}{I(t)} \geq \epsilon.
    \end{equation}
\end{definition}

The derivation of this $C$ and its dependence on $\gamma$, in the case of the temporal version of \eqref{eq:model1}, remains open.

\section{Numerical Simulations}
In this section, we numerically explore the spatiotemporal patterns generated by model \eqref{eq:model1}. The PDEPE computing package in Matlab is implemented to solve the system \eqref{eq:model1} in one dimension. Here we take $r=5,K=75, \lambda=3, m=70.8, 
a=12,\mu=3.4,d=8.3, \delta_1=1, \delta_2=1, \delta_3=2, \Omega = (0,30\pi)$(one-dimensional space) and we consider $\xi(P)=P, \eta(P)= P/(1-P)$. In this case, $\lambda K > \mu$ and $m > d + \lambda ad(r+\lambda K)/r(\lambda K - \mu)$, it can calculate that $\mathcal{E}^*=(1.695,13.83,0.6147)$ is the unique positive constant steady state solution. Then all non-negative solutions converges to $\mathcal{E}^*$ for $\chi_1=\chi_2=0$ from Fig \eqref{fig1}, where the initial value is $(0.1\text{cos}(x)+0.35;0.1\text{sin}(2x)+0.6;3\text{cos}(3x)+1.6)$. Inclusion of small taxis $\chi_1=\chi_2=10$ does not destroy global stability of $\mathcal{E}^*$ (see Fig.\eqref{fig2}), whereas the large taxis components $\chi_1=\chi_2=100$ will destroy the global stability of $\mathcal{E}^*=(1.695,13.83,0.6147)$ and enhance the spatial pattern formations (see Fig \eqref{fig3}). Further, when $\chi_1 > \chi_2$, it fulfills the stability condition \eqref{eq:main1} when considering specific numerical values for the remaining parameters. In such case, periodic solution with fluctuations emerges for the same initial values (see Fig \eqref{fig4}). Moreover, when $\chi_2 > \chi_1$, it fails to be stable and enhance the spatial pattern (see Fig \eqref{fig5}). 
\begin{figure}[hbt!]
    \centering
    \includegraphics[width=18cm]{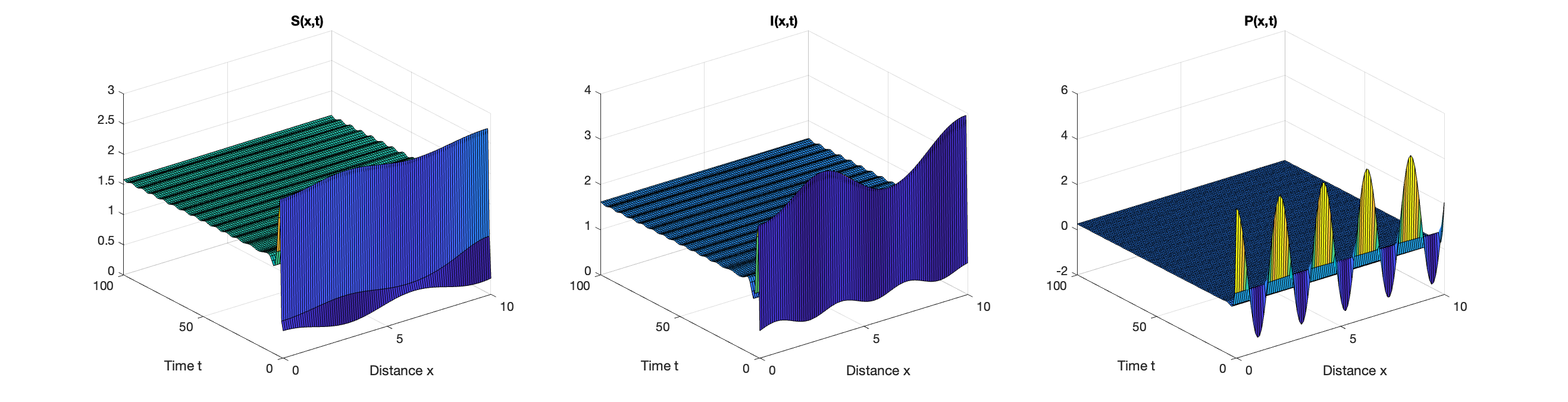}
    \caption{$\mathcal{E}^*$ remains stable for \eqref{eq:model1} when  $\chi_1=\chi_2=0$}
    \label{fig1}
\end{figure}
\begin{figure}[hbt!]
    \centering
    \includegraphics[width=18cm]{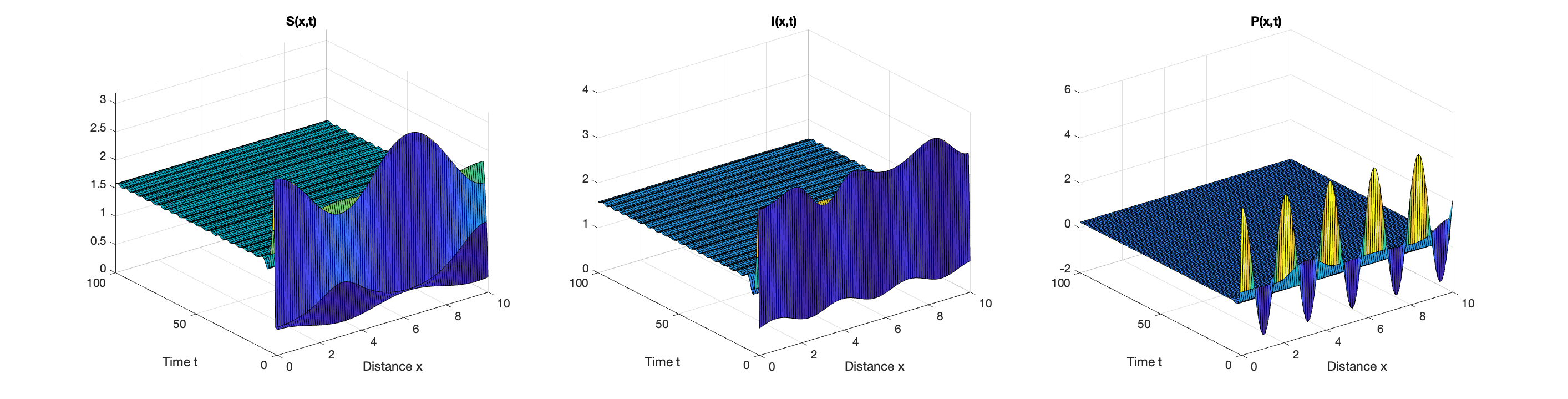}
    \caption{$\mathcal{E}^*$ remains stable for \eqref{eq:model1} for small taxis $\chi_1=\chi_2=10$}
    \label{fig2}
\end{figure}
\begin{figure}[hbt!]
    \centering
    \includegraphics[width=18cm]{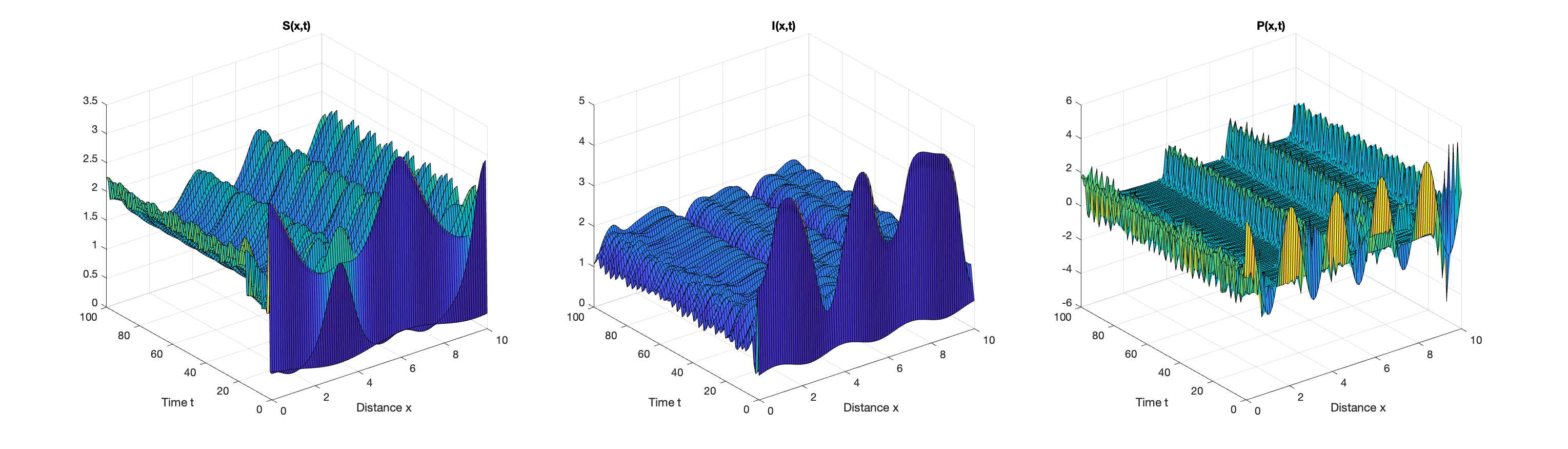}
    \caption{Instability of $\mathcal{E}^*$ induced by large taxis $\chi_1=\chi_2=100$}
    \label{fig3}
\end{figure}
\begin{figure}[hbt!]
    \centering
    \includegraphics[width=18cm]{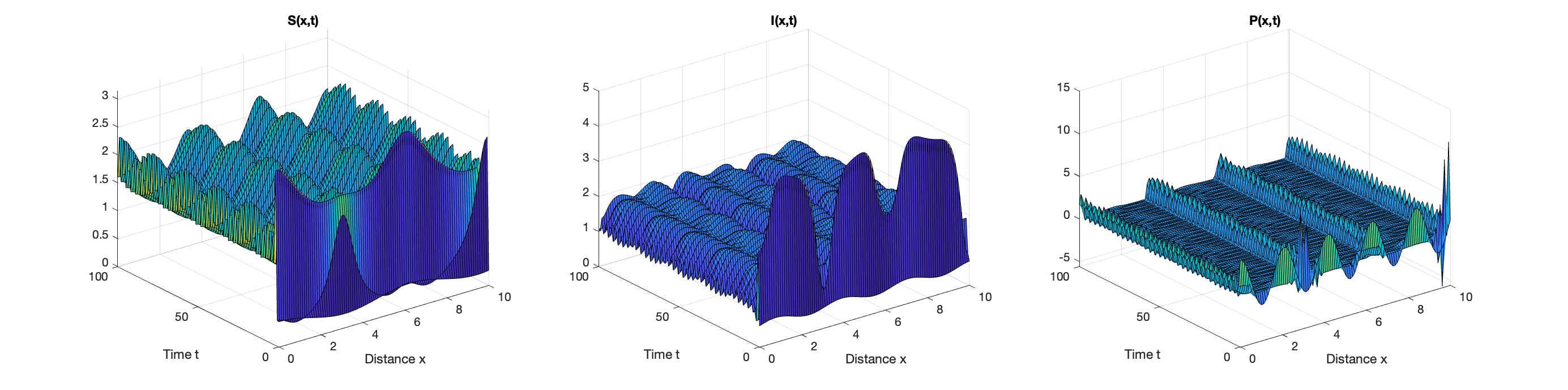}
    \caption{Induced periodicity with some fluctuation when $\chi_1=100, \chi_2=40$}
    \label{fig4}
\end{figure}
\begin{figure}[hbt!]
    \centering
    \includegraphics[width=18cm]{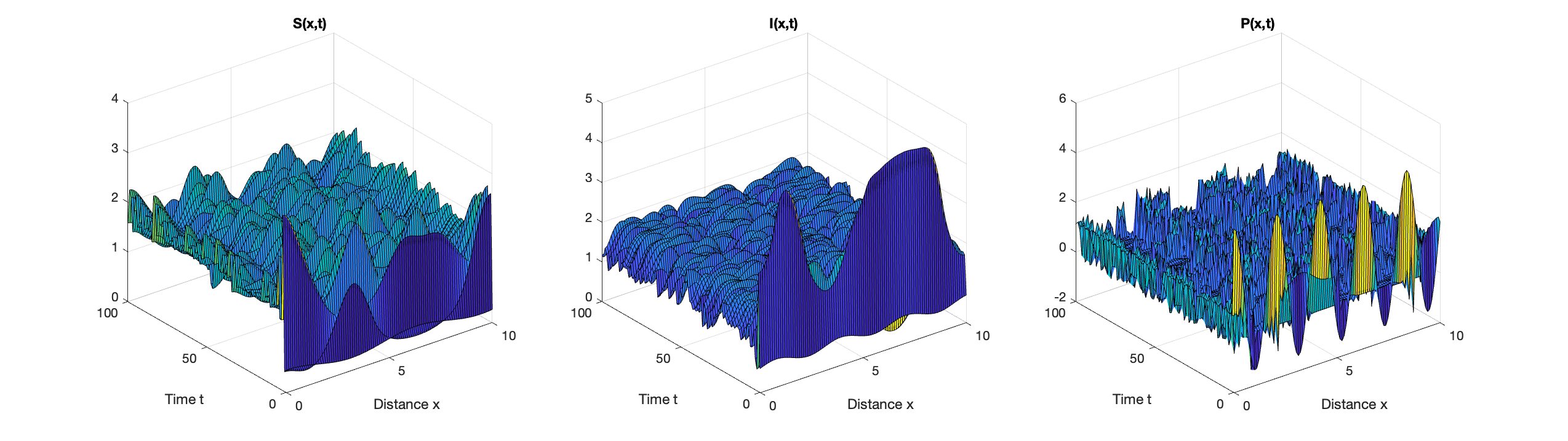}
    \caption{Irregular patterns occur when $\chi_1=40, \chi_2=100$}
    \label{fig5}
\end{figure}
\newpage
\section{Discussion and Conclusion}
Ecological interactions characterizing the behaviour of participating species in a domain are rather complex and cannot be described solely by predator-prey dynamics, which are based on the premise that the only interaction occurring in the habitat is that of the predator feeding on the prey. There are several other examples of relationships in which species depend on each other for survival, along with many other forms of interactions that shape the dynamics of populations and communities. Understanding these interactions is crucial for predicting how ecosystems respond to changes and disturbances. Although ecologically more realistic, such complex interactions are less frequently modeled due to the inherent difficulty in formulating a mathematical characterization of these interactions. Moreover, even when they are modeled, providing a thorough mathematical foundation remains a challenge due to the complex nonlinearities that govern these interactions. An important class of such interactions is represented by eco-epidemic models, which capture the interplay between ecological and epidemiological dynamics. In these models, a particular species contracts a disease, and in addition to predation, the disease also influences its behavior. Moreover, the disease also affects the feeding behaviour of other participating species, as well as their spatial dynamics.

Though ecologically and realistically relevant, the mathematical literature on spatiotemporal eco-epidemic remains comparatively limited. To the best of our knowledge, there have been very no studies that consider the role of complex nonlinear diffusion in such models. Our work aims to bridge this gap. With this motivation, we study an eco-epidemic model that incorporates slow diffusion for the infected prey class, along with prey-taxis directed towards both the susceptible and infected prey populations.

 The mathematical analysis of the model has been carried out in two cases:
    \begin{itemize}
    \item In the first case, we assume that all participating species undergo linear diffusion characterizing random movement in the domain, and the predator exhibits attractive taxis via which its motility is directed towards the area of prey abundance. Under the assumption of a so-called \textit{volume-filling effect} for ecological systems, which is standard for such interactions, we establish the global existence of classical solutions for this scenario. 
    \item In the second case, we focus on a scenario where the infected prey undergoes slow diffusion governed by the $p$-Laplacian operator. Additionally, its growth is inhibited by a modified mortality term. We prove the global existence of weak solutions for this scenario. The inclusion of these complexities further divides the analysis into two subcases, and we demonstrate that when the infected prey experiences nonlinear absorption, finite-time extinction of the infected class can occur.
\end{itemize}
After conducting an extensive mathematical analysis of the proposed model, we turn our attention to the stability analysis of the positive equilibrium in the case where all species undergo linear diffusion, the mortality of the infected prey is normal, and the predator’s taxis is directed solely towards the infected prey. We prove the existence of steady-state bifurcation for this system using standard bifurcation theory, specifically in the context of constructing a zero-order Fredholm operator for the elliptic system.
Our study tends to focus on the understanding of complex ecological interactions via the formulation and analysis of nonlinear partial differential equation systems, incorporating mechanisms such as nonlinear diffusion, cross-diffusion (taxis), and density-dependent mortality, to capture realistic ecological and epidemiological processes. 

Our results have several realistic implications that can be interpreted in the context of aphid invasions in agricultural crops incorporating FTE mechanism:
\begin{itemize}
    \item  The aphid is a primary insect pest on many agricultural crops \cite{dedryver2010conflicting}. The soybean and cotton aphids in particular cause large scale damages to agricultural yield in the United States \cite{ragsdale2011ecology}.
    \item Although integrated pest management programs promote the use of classical top down biological control, such as predators and parasitoids, the use of disease is less explored. However, there is evidence that fungicides negatively effect aphid populations \cite{koch2010non} - and that these populations get slower and more lethargic due to fungicide applications, as well as suffer enhanced mortality. Thus, future models for aphid control can draw from \eqref{eq:model1}, with a combination of fungicidal applications. Here one can assume predators such as the asian beetle (\emph{Harmonia axyrids}) will feed on the infected aphid, but chase after both susceptible and infected aphids. Further, the fungicide application will slow down the infected aphid and enhance their mortality. Further Theorem \ref{thm:FFTEsd} tells us that extinction of the infected aphid could be possible, and Lemma \ref{lem:l1da} describes what parametric and initial data restrictions this is possible for.
\end{itemize}
Our results also have several realistic implications in the context of arthopod invasions \cite{brown2011global}. These include the invasion of the chief predator of the Soybean aphid, the Asian beetle (\emph{Harmonia axyrids}), in the Unites States \cite{koch2008bad}. 
\begin{itemize}
    \item The multicolored Asian Beetle (\emph{Harmonia axyrids}) was first detected in the North-Central US back in the mid 90's. It is the chief predator of the Soybean aphid and has been able to suppress aphid populations in laboratory conditions. However, aphid populations will still reach very high densities during the summer months \cite{ragsdale2011ecology, ragsdale2004soybean}. Current investigations have focused on regime shifts, as well as integrated pest management programs that include competing enemies of the aphid, such as parasitoids \cite{bahlai2015shifts, ragsdale2011ecology, o2018rapid}. 

    \item However, current work has underexplored the impact of disease in the aphid on the beetle densities, and the resulting population dynamics as well as possible regime shifts. This further alludes to larger questions of the paradox of enrichment and stabilization of populations under a possible combination, if factors such as insecticide use, as well as enemy competition \cite{raffel2008parasites} - where now disease in pest could also play a key role in stabilization, as well as in pest extinction, as seen via Theorem \ref{thm:FFTE} and Theorem \ref{thm:FFTEsd}.
\end{itemize}

All in all our work has implications for both biological invasions as well as the biological control of invasive pests, when disease in the pest/prey is considered. Future work can look at dispersal mechanisms, both fast and slow, driven by porous medium type equations, instead of the $p$-Laplace type. Herein further novel dynamics could be exhibited.

\subsubsection*{Acknowledgement}
 NMT acknowledges the Ministry of Education, Government of India, for the Prime Minister’s Research Fellowship
and Grant (PMRF ID: 1602140) for providing the necessary funding. NM acknowledges funding by the German Research Foundation DFG within the SPP 2311 (Grant No.- 679005).


\end{document}